\documentclass[12pt]{article}
\usepackage{amssymb}
\usepackage{graphicx}
\usepackage[all]{xy}
\usepackage{color}
\usepackage{latexsym}
\usepackage{amssymb,amsmath,amstext}
\usepackage{amsthm}
\usepackage{epsfig}
\usepackage{vmargin} 
\setpapersize{USletter} 
\setmargrb{1in}{1in}{1in}{1in} % --- sets all four margins.  Cool. 
\numberwithin{equation}{section} 
\newtheorem{theorem}{Theorem}[section] 
\newtheorem{proposition}[theorem]{Proposition} 
\newtheorem{lemma}[theorem]{Lemma} 
\newtheorem{corollary}[theorem]{Corollary} 
 
\newtheorem{defn}[theorem]{Definition}

\theoremstyle{definition}

\newtheorem{example}[theorem]{Example} 
\newtheorem{remark}[theorem]{Remark}

\newcommand{\R}{\mathbb{R}}
\newcommand{\C}{\mathbb{C}}

\newcommand{\PP}{\mathbb{P}}

\date{}

\begin{document} 
\title{\bf Quartic Spectrahedra}

\author{
John Christian Ottem, Kristian Ranestad, \\
Bernd Sturmfels and Cynthia Vinzant}

\maketitle

\begin{abstract}
\noindent
Quartic spectrahedra in $3$-space form a 
semialgebraic set of dimension $24$. 
This set is stratified by the location of the ten nodes
of the corresponding real quartic surface.
There are twenty maximal strata, 
identified recently by Degtyarev and Itenberg,
via the global Torelli Theorem for real K3 surfaces.
We here give a new proof that is self-contained and algorithmic.
This involves extending
 Cayley's  characterization of quartic symmetroids,
 by the property that the branch locus of the projection from a node
consists of two cubic curves. 
 This paper represents a first step towards the classification of
 all spectrahedra of a given degree and dimension.
      \end{abstract}
      
\section{Introduction}

Spectrahedra are fundamental objects in convex algebraic geometry \cite{BPT}.
A spectrahedron is the intersection of the cone of positive-semidefinite matrices
with an affine subspace of the space of real symmetric $n \times n$-matrices.
In this paper, the subspace has dimension three and is identified with $\R^3$.
The algebraic boundary of a three-dimensional spectrahedron in $\R^3$
is a surface $V(f)$ in complex projective space $\C \PP^3$.
Its defining polynomial
$f(x) = f(x_0,x_1,x_2,x_3)$ 
is the determinant of a symmetric matrix $A(x)$ of linear forms.
Explicitly,
\begin{equation}
\label{eq:AxAx}
 A(x) \,\,\, = \,\,\, A_0 x_0 + A_1 x_1 + A_2 x_2 + A_3 x_3 . 
 \end{equation}

The surfaces $V(f)$ are known as {\em symmetroids}
in classical algebraic geometry \cite{Ca1, Co, Cos, J}. 
Generally, we  allow the entries in the symmetric $n \times n$-matrices
$A_0,A_1,A_2,A_3$ to be arbitrary complex numbers.
We say that $V(f)$ is a {\em real symmetroid} if
all $\binom{n+3}{3}$ coefficients of $f$ are  real numbers and that
$V(f)$ is a {\em very real symmetroid} if
$A_0,A_1,A_2,A_3$ can be chosen to have real entries.
The {\em spectrahedron} associated to a very real symmetroid $V(f)$ is the set
\[ S(f) \,\,\, = \,\,\,
\bigl\{ x \in \R \PP^3 \,:\,
A(x)=A_0 x_0 + A_1 x_1 + A_2 x_2 + A_3 x_3
\,\,\,\hbox{is semidefinite}  \}.
\]
 The constraint is denoted $A(x) \succeq 0$ and is called a
 {\em linear matrix inequality} (LMI). 
 Maximizing a linear function over the spectrahedron $S(f)$ is a {\em semidefinite program}. 
 In optimization theory, one often considers the spectrahedron $S(f)$ is an affine hyperplane, e.g. $\{x_0=1\}$. 
 While much of the underlying geometry does not change, it is more convenient for us to use 
 the language of projective geometry, rather than work in affine space. 
 Using G\aa rding's theory of {\em hyperbolic polynomials} \cite{G, Ren},
one sees that $S(f)$ is determined by $f(x)$, i.e.~the 
matrices $A_0,A_1,A_2,A_3$ are not needed to identify the set $S(f)$.
  We say that a symmetroid $V(f)$ is {\em spectrahedral} if it is very real and $S(f)$ is
   full-dimensional in $\R \PP^3$.

A symmetroid is called {\em nodal} if all its singular points are nodes, i.e.~isolated quadratic singularities.  
Generically there are ${n+1 \choose 3}$  nodes on $V(f)$. 
A point $x\in \R\PP^3$ where the matrix $A(x)$ has rank $k$ will be called
a {\em rank-$k$-point}.
 Such a point is  singular on the surface $V(f)$ and generically it is a node. 
 A nodal symmetroid has exactly $\binom{n+1}{3}$  rank-$(n-2)$ points \cite{HT} and is called {\em transversal} if it does not have any further nodes.
% A nodal symmetroid is {\em transversal} if the number of nodes equals $\binom{n+1}{3}$, in which case all the nodes 
% are rank-$(n-2)$ points.
A spectrahedron $S(f)$ is {\em nodal} (resp.~{\em transversal}) if its symmetroid
$V(f)$ has this property. For example, the {\em Kummer symmetroid} in 
Figure \ref{fig:kummer} is nodal but not
transversal: it has $16$ nodes, not just~$10$.

The set  ${\cal S}$ of symmetroids forms an irreducible 
variety in the projective space $\C \PP^{\binom{n+3}{3}-1}$ of
all surfaces of degree $n$ in $\C \PP^3$, and a generic point in ${\cal S} $ corresponds to a transversal symmetroid.
The set ${\cal S}_{\rm spec}$ of spectrahedral symmetroids  is
Zariski dense in the variety ${\cal S}$ of complex symmetroids. The objects 
above form a nested sequence of semialgebraic subsets:
\begin{equation}
\label{eq:inclusions} {\cal S}_{\rm spec} \,\,\subseteq \,\,
{\cal S}_{\rm veryreal}  \,\,\subseteq \,\,
{\cal S}_{\rm real}  \,\,= \,\, 
{\cal S} \, \cap \, \R \PP^{\binom{n+3}{3}-1}. 
\end{equation}
For $n = 1$, every spectrahedron is a halfspace. 
For $n=2$, a spectrahedron is a quadratic cone and the left inclusion in (\ref{eq:inclusions}) is an equality.
The quadric $x_0^2 + x_1^2 + x_2^2$  lies in ${\cal S}_{\rm real} \backslash {\cal S}_{\rm veryreal}$,
because no triple of real symmetric $2 \times 2$-matrices satisfies
 ${\rm det}(A_i) = 1 $ and ${\rm det}(A_j+A_k) = 2$.

\vskip -0.42cm
\begin{center}
\begin{figure}[h]
\vskip -0.42cm
\begin{center} 
\includegraphics[height=1.8in]{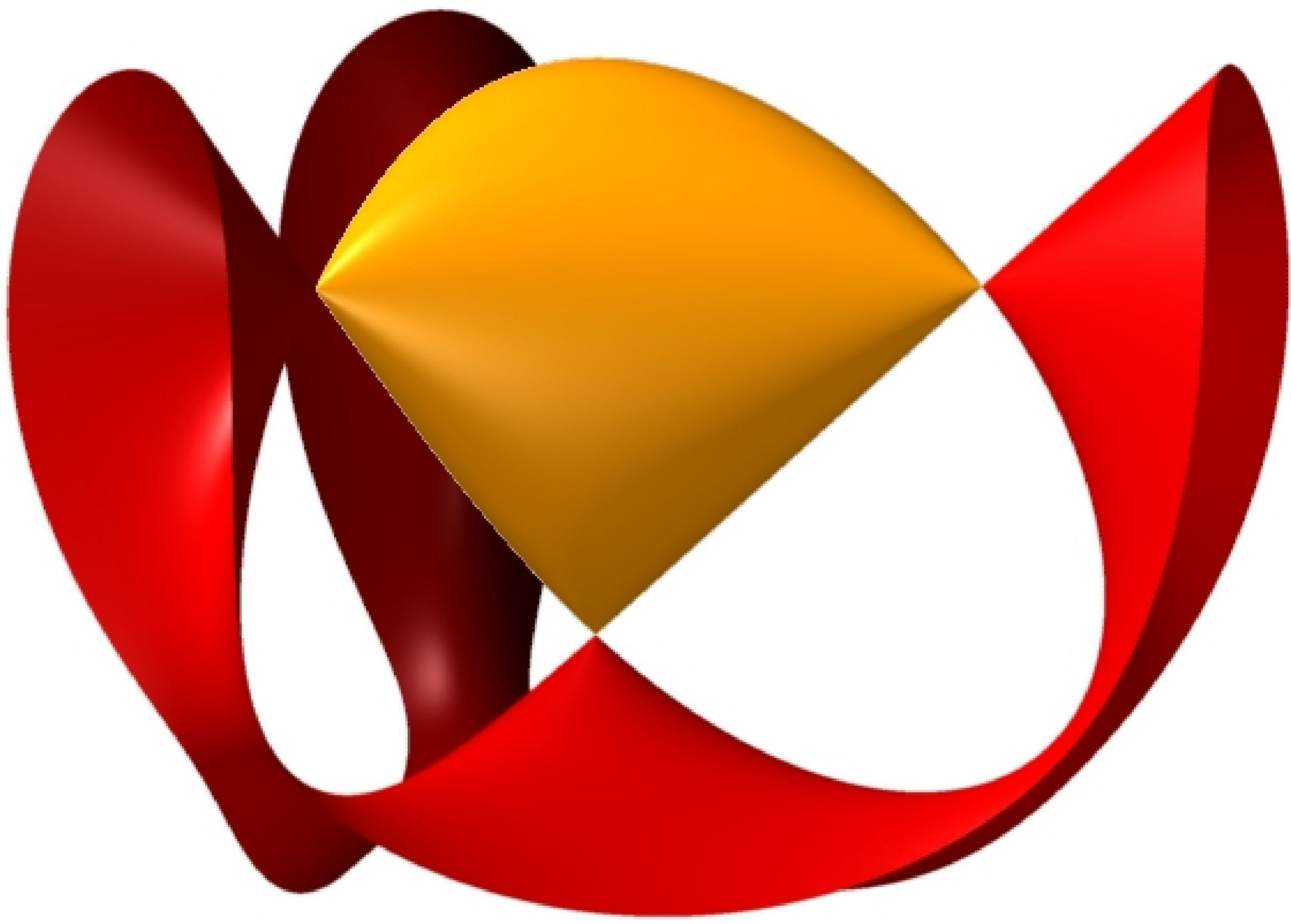}
 \qquad 
\includegraphics[height = 1.7in]{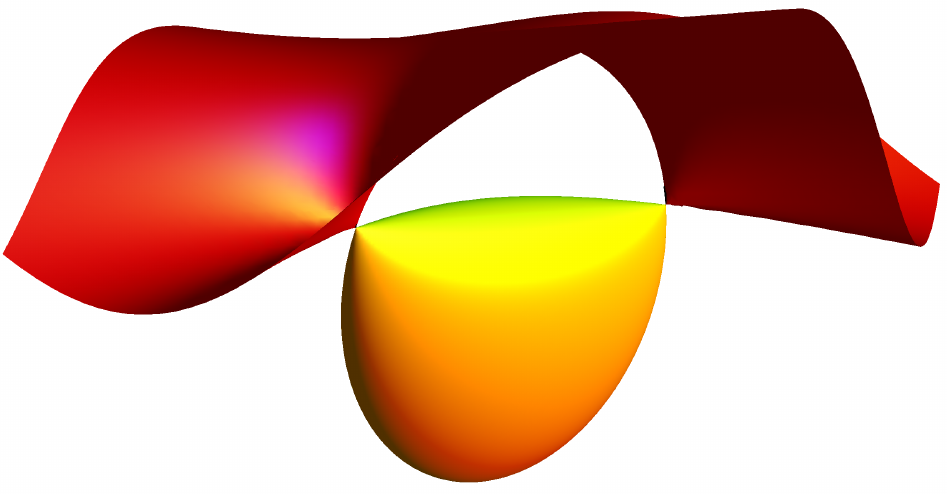}
\end{center}
%\vskip -0.4cm
\caption{\label{fig:redyellow} 
A transversal cubic spectrahedron has either four nodes or two nodes.
}
\end{figure}
\end{center}
\vskip -0.4cm

The case $n=3$ of cubic spectrahedra is  visualized by the two pictures in Figure \ref{fig:redyellow}.
Here the algebraic boundary is Cayley's {\em cubic symmetroid}. Such a cubic surface
has four complex nodes, and it is unique up to projective
transformations over $\C$. It is obtained
by blowing up $\C \PP^2$ at the six intersection points of four general lines.
Cubic symmetroids in $\mathcal{S}_{\rm real}$ correspond to configurations of four lines
that can be defined collectively over $\R$. There are three  types:
\begin{itemize}
\item[(a)] All four lines are defined over $\R$. This gives the left picture with four real nodes.
\item[(b)] Two lines are defined over $\R$ and the others form a  complex conjugate pair of lines.
This gives the right picture, with two real nodes, both in the boundary of the spectrahedron.
\item[(c)] There are two complex conjugate pairs of lines.
\end{itemize}
The two first types  are shown in Figure~\ref{fig:redyellow}. 
Very real cubic symmetroids of type (c) have no real nodes,
 and they belong to the set $\mathcal{S}_{\rm veryreal} \backslash \mathcal{S}_{\rm spec}$ (cf. Remark \ref{cubicspec} and Example \ref{nonspec}).

In this paper we investigate the case $n=4$ of quartic spectrahedra.
The corresponding variety $\mathcal{S}$ has codimension $10$
in the $\C \PP^{34}$ of all quartics. In the context of
convex algebraic geometry it has appeared
in \cite[Theorem 3]{BHORS}, where it parametrizes extremal
non-negative quartics that are not sums of squares.
Such symmetroids, like the Choi-Lam-Reznick quartic in \cite[(8)]{BHORS},
are in $\mathcal{S}_{\rm real} \backslash \mathcal{S}_{\rm veryreal}$.
See Example \ref{nonspec} for a quartic symmetroid
in  $\mathcal{S}_{\rm veryreal} \backslash \mathcal{S}_{\rm spec}$.

In response to a question by the third author,
Degtyarev and Itenberg \cite{DI} studied
the location of the ten nodes on a spectrahedral symmetroid.
They proved the following result:

\begin{theorem} 
\label{thm:DI}
There exists a transversal quartic spectrahedron
with $\sigma$ nodes on its boundary and $\rho$ real nodes in its symmetroid 
if and only if $0 \leq \sigma \leq \rho $, both are even, and $2 \leq\rho \leq 10$.
 \end{theorem}
 
 The proof given by Degtyarev and Itenberg in \cite{DI} is very indirect.
It  translates the problem into
the period domain for K3 surfaces via the Global Torelli Theorem,
and it proceeds by characterizing the various strata in $\mathcal{S}$
in terms of intersection lattices.
That approach cannot be used to
actually construct matrices $A(x)$ for
 quartic spectrahedra.

We here present a new proof  which
is elementary, computational, and self-contained.
For the if-direction of Theorem~\ref{thm:DI}, 
Section 2 exhibits twenty  matrices $A(x)$ which realize the allowed pairs of parameters $(\rho,\sigma)$.
That list generalizes Figure \ref{fig:redyellow}, which 
depicts the  classification of cubic
spectrahedra into two types,
$(\rho,\sigma) = (2,2) \, {\rm and} \,  (\rho,\sigma) = (4,4)$.

Our proof of the only-if direction uses
classical projective geometry, and it will
be given in Section 4. The idea is to build an
 explicit parametrization of the semi-algebraic
set $\mathcal{S}_{\rm spec}$ in terms of the 
ramification data of the projection of a symmetroid from one of its nodes.

The projection of a quartic surface from a node $p$ is a double cover of $\C \PP^2$
ramified along a sextic plane
 curve $R_p$  with a totally tangent conic $C_p$, the image of the tangent cone at $p$.  
This picture goes back to Cayley \cite{Ca1}
who used it to characterize nodes on quartic symmetroids. 
We extend Cayley's result to the real numbers, with focus on spectrahedra.
 
\begin{theorem} 
\label{thm:cayley}
Let $p$ be a node on a quartic surface $X$ in $\C \PP^3$. Then $X$ is a symmetroid if and only if the ramification curve is the union of two cubic curves $R_p=R_1\cup R_2$. Suppose that $X$ contains no line through $p$. Then $X$ is transversal if and only if $R_1$ and $R_2$ are smooth, intersect transversally in 9 points, and the conic $C_p$ is smooth. In this case, the surface $X$ and the point $p$ are real if and only if 
the curves $R_1\cup R_2$ and $C_p$ are real. Furthermore,
\begin{enumerate}
\item  $X$ is very real and $p$ is real if and only if $R_1\cup R_2$ is real and $C_p$ has a real point. 
\item  $X$ is spectrahedral and $p$ lies on the spectrahedron if and only if 
the cubics $R_1$ and $R_2$ 
are complex conjugates  and  the conic $C_p$ has a real point.
\item  $X$ is very real and $p$ is a real node that does not lie on the spectrahedron if and only if 
the cubics $R_1$ and $R_2$ are real  and the conic $C_p$ has a real point. 
\end{enumerate}
\end{theorem}

Theorem \ref{thm:cayley} characterizes the location
of generic points in the various strata of (\ref{eq:inclusions}).
Special points are obtained as limits of these.
A special symmetroid may contain a line that
passes through a node $p$, and the
cubics $R_1, R_2$ may be singular.
For instance, the Kummer symmetroid
shown in  Figure \ref{fig:kummer}
contains $16$ lines
and each cubic $R_i$ factors into three lines.

\smallskip

The paper is organized as follows.
Section 2 contains our gallery of transversal quartic spectrahedra.
Section 3 develops the projective geometry of (real) quartic symmetroids 
and their projections from nodes. It furnishes a self-contained proof
of Theorem~\ref{thm:cayley}.
Section 4 completes the proof of Theorem \ref{thm:DI}
and develops a further extension of 
Theorem~\ref{thm:cayley} in terms of interlacing cubics.
Section 5 serves ``spectrahedral treats'':
  various families of quartics seen
in the convex algebraic geometry literature.
We explain how they fit into our general theory.

\section{Twenty Types of Transversal Spectrahedra} 

A transversal quartic symmetroid is a surface in $\C \PP^3$ with $10$ nodes. We are primarily interested in real symmetroids, whose defining polynomial is defined over $\R$.
A transversal quartic spectrahedron is a convex body in $\R \PP^3$ whose
(topological) boundary has a transversal quartic symmetroid as its Zariski closure.
The {\em type} $(\rho, \sigma)$ of a transversal quartic spectrahedron
records the number $\rho$ of real nodes of the symmetroid
and the number $\sigma$ of nodes that lie on the spectrahedron.
Since the spectrahedron's boundary is part of the real symmetroid, we always have
 $\sigma \leq \rho$.
 
 \vskip -0.5cm
\begin{figure}[h]
\begin{center}
\includegraphics[height=2.25in]{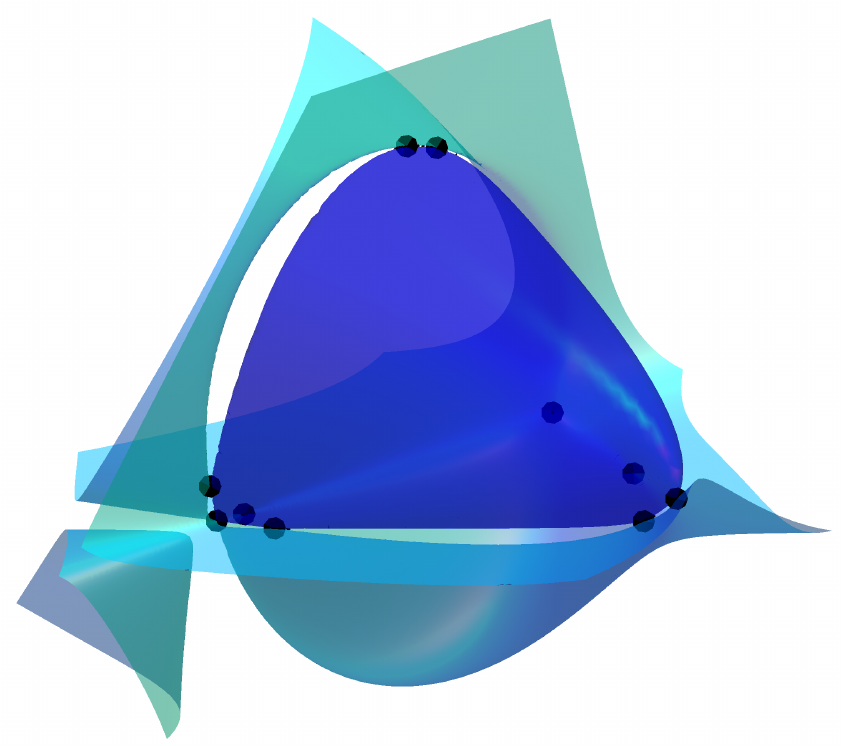} \quad  \quad  \includegraphics[height=2.25in]{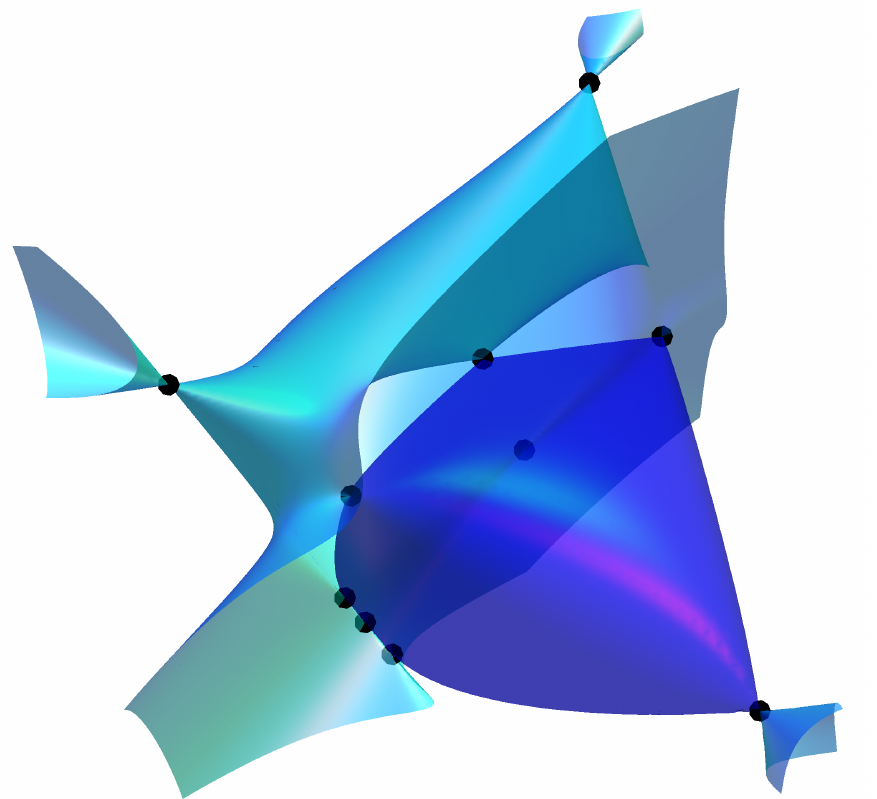} 
\vskip -0.5cm
\caption{ \label{fig:ex_10_10}
Two quartic spectrahedra, of type $(\rho,\sigma) = (10,10)$ and of type
$(\rho,\sigma) = (10,8)$.}
\end{center}
\vskip -0.7cm
\end{figure}

Figure \ref{fig:ex_10_10} illustrates the situation.
Spectrahedral symmetroids have one distinguished
connected component in their complement in $\R \PP^3$,
namely the spectrahedron itself. This is a convex body
which serves as the set of feasible points in
{\em semidefinite programming} (SDP) \cite{BPT}.
Both symmetroids in Figure \ref{fig:ex_10_10} have all $10$ nodes real.
The one on the left has all $10$ nodes on the spectrahedron.
The symmetroid  on the right has $8$ nodes on the spectrahedron
and $2$ nodes off the spectrahedron.
Concretely, the spectrahedron on the left is the compact convex set 
defined by requiring that the
following matrix $A(x)$ is positive semidefinite:
$$
{\scriptsize
\begin{bmatrix}
45 x_0 + 40 x_1 + 65 x_2 + 72 x_3 \! & \!    33 x_0 + 50 x_1 + 7 x_2 - 60 x_3 \! & \! -84 x_0 + 12 x_1 + 3 x_2 - 54 x_3 \! & \! -21 x_0 + 16 x_1 + 54 x_2 -  18 x_3\\
33 x_0 + 50 x_1 + 7 x_2 - 60 x_3 \! & \! 82 x_0 + 85 x_1 + 2 x_2 + 82 x_3 \! & \! -99 x_0 + 12 x_1 + 6 x_2 + 25 x_3 \! & \! -46 x_0 + 41 x_1 + 18 x_2 +   51 x_3\\
-84x_0 + 12 x_1 + 3 x_2 - 54 x_3 \! & \! -99 x_0 + 12 x_1 + 6 x_2 + 25 x_3 \! & \!  181 x_0 + 4 x_1 + 26 x_2 + 53 x_3 \! & \! 59 x_0 + 2 x_1 + 58 x_2 - 9 x_3\\
-21 x_0 + 16 x_1 + 54 x_2 - 18 x_3 \! & \! -46 x_0 + 41 x_1 + 18 x_2 + 51 x_3 \! & \! 59 x_0 + 2 x_1 + 58 x_2 - 9 x_3 \! & \! 26 x_0 + 26 x_1 + 164 x_2 + 45 x_3
\end{bmatrix}.}
$$

In the optimization literature, this constraint is denoted $A(x) \succeq 0$ and is 
called a linear matrix inequality (LMI).
Maximizing a linear function over
the spectrahedron, the optimal solution $A(x)$ has
either rank $3$ and is attained at a smooth point
in the boundary, or it has rank $2$ and is attained at one
of the nodes. For the types with $\sigma = 0$, 
none of the nodes lie in the boundary 
of the spectrahedron. In this case,
the rank of the optimal solution
to the SDP is always $3$, for any choice of linear cost function.
Here is a specimen.

\begin{figure}[h]
\vskip -0.3cm
\begin{center}
\parbox[m]{3in}{\includegraphics[height=3in]{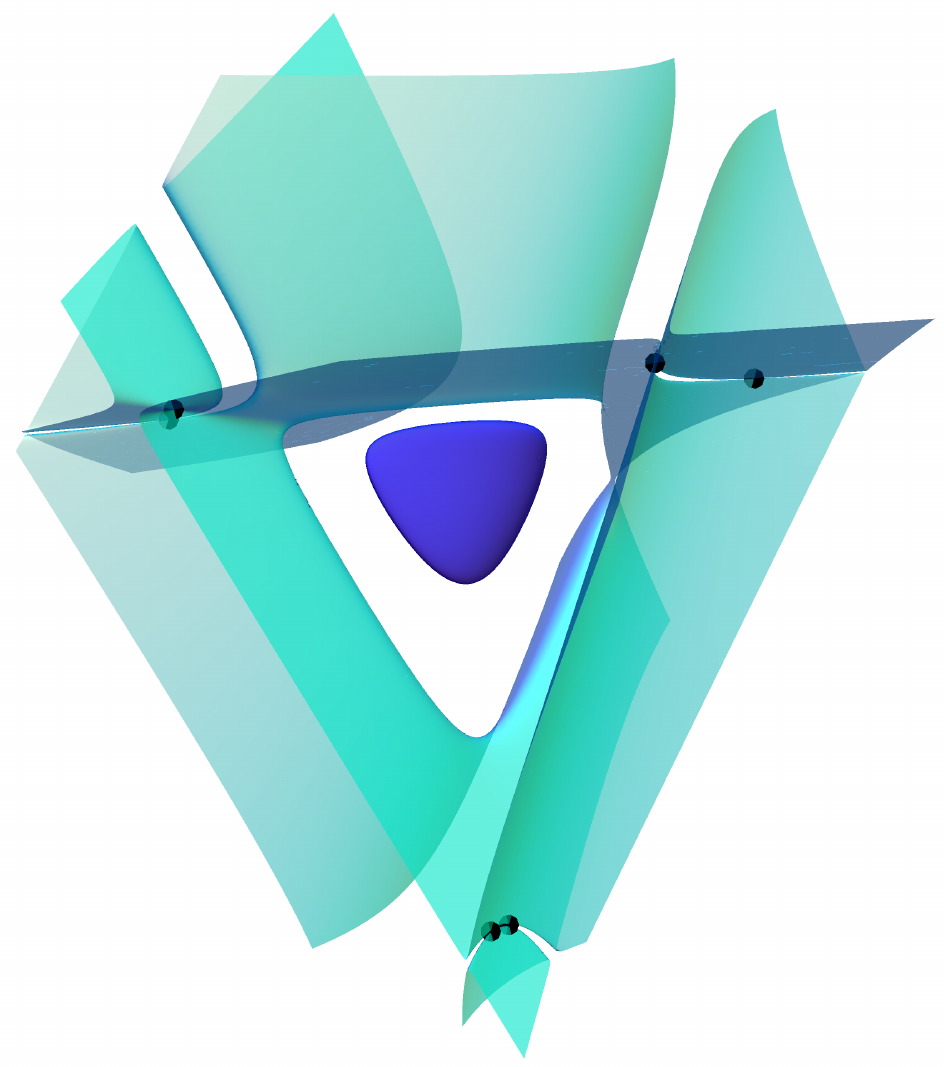} }
\parbox[m]{2in}{\includegraphics[width=2in]{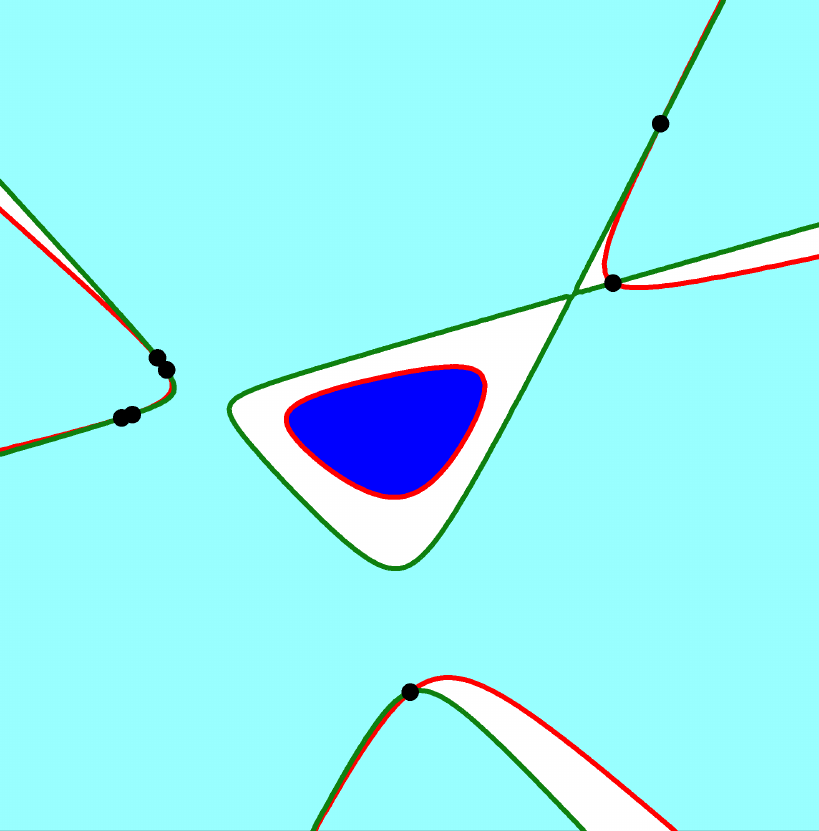} }
\end{center}
\vskip -0.5cm
\caption{\label{fig:ex10_0}
A quartic spectrahedron of type $(\rho,\sigma) = (10,0)$, and its projection from a node.
}
\end{figure}
\vskip -0.5cm

\begin{example}
Figure \ref{fig:ex10_0} shows a quartic
symmetroid with $10$ real nodes.
None of them lie on the spectrahedron.
The topological boundary of the spectrahedron
is a smooth $2$-sphere $\mathbb{S}^2$.
As argued in Figure \ref{fig:redyellow} and
Remark \ref{cubicspec},
 no such smooth spectrahedra exist for $n=3$. 
$\hfill \diamondsuit$
\end{example}

Here is one more contribution to our picture gallery.
Figure \ref{fig:ex_2_2} shows a quartic spectrahedron that has
two nodes in its boundary, and no further real nodes on the symmetroid.
In this diagram, the real part of the symmetroid is a compact surface
in affine $3$-space.
\begin{figure}[h]
\begin{center}
\vskip -0.2cm
\includegraphics[scale=0.9]{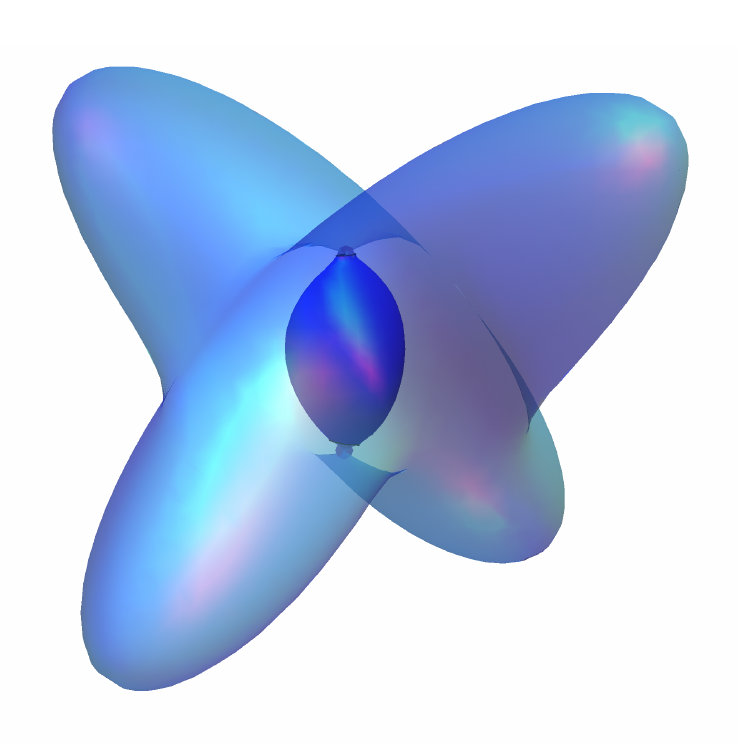}  
\vskip -0.4cm
\caption{ \label{fig:ex_2_2}
A quartic symmetroid of type $(\rho,\sigma) = (2,2)$.}
\end{center}
\end{figure}
\vskip -0.4cm

We now establish the existence part of the Degtyarev-Itenberg theorem.

\begin{proof}[Proof of the if direction in Theorem \ref{thm:DI}]
The constraints $0 \leq \sigma \leq \rho \leq 10$ and $\rho \geq 2$
allow for $20$ solutions $(\rho,\sigma)$ among even integers.
In the following table we list the twenty pairs $(\rho, \sigma)$ followed
by four symmetric $4 \times 4$-matrices   $A_0,A_1,A_2,A_3$ 
with integer entries. Each  quadruple specifies a matrix
$A(x)$ as in (\ref{eq:AxAx}) whose symmetroid $f = {\rm det}(A(x))$
is transversal and has a non-empty spectrahedron $S(f)$.
To verify the correctness of the list,
one computes the ten complex nodes, one 
checks that $\rho$ of them are real, and
one examines how many lie on the
spectrahedron. The latter test is done
by computing the eigenvalues of the
matrix $A(x)$ at each node $x$.
If all eigenvalues have the same sign 
then $x$ is on the spectrahedron.
The list starts with $(\rho,\sigma) = (2,2)$, as in Figure \ref{fig:ex_2_2},
and ends with $(\rho, \sigma) = (10,0)$, as in Figure~\ref{fig:ex10_0}.

{ \footnotesize
% \begin{small}
$$
\begin{matrix}
(2,2) \, : \! & \! 
\begin{bmatrix}
 3 \! & \! 4 \! & \! 1 \! & \! -4 \\
 4 \! & \! 14 \! & \! -6 \! & \! -10 \\
 1 \! & \! -6 \! & \! 9 \! & \! 2 \\
 -4 \! & \! -10 \! & \! 2 \! & \! 8 \\
\end{bmatrix} \! & \! \begin{bmatrix}
 11 \! & \! 0 \! & \! 2 \! & \! 2 \\
 0 \! & \! 6 \! & \! -1 \! & \! 4 \\
 2 \! & \! -1 \! & \! 6 \! & \! 2 \\
 2 \! & \! 4 \! & \! 2 \! & \! 4 \\
\end{bmatrix} \! & \! \begin{bmatrix}
 17 \! & \! -3 \! & \! 2 \! & \! 9 \\
 -3 \! & \! 6 \! & \! -4 \! & \! 1 \\
 2 \! & \! -4 \! & \! 13 \! & \! 10 \\
 9 \! & \! 1 \! & \! 10 \! & \! 17 \\
\end{bmatrix} \! & \! \begin{bmatrix}
 9 \! & \! -3 \! & \! 9 \! & \! 3 \\
 -3 \! & \! 10 \! & \! 6 \! & \! -7 \\
 9 \! & \! 6 \! & \! 18 \! & \! -3 \\
 3 \! & \! -7 \! & \! -3 \! & \! 5 
\end{bmatrix} \smallskip \\
 (4 , 4) 
\, : \! & \!
\begin{bmatrix}
 18 \! & \! 3 \! & \! 9 \! & \! 6 \\
 3 \! & \! 5 \! & \! -1 \! & \! -3 \\
 9 \! & \! -1 \! & \! 13 \! & \! 7 \\
 6 \! & \! -3 \! & \! 7 \! & \! 6 \\
\end{bmatrix} \! & \! \begin{bmatrix}
 17 \! & \! -10 \! & \! 4 \! & \! 3 \\
 -10 \! & \! 14 \! & \! -1 \! & \! -3 \\
 4 \! & \! -1 \! & \! 5 \! & \! -4 \\
 3 \! & \! -3 \! & \! -4 \! & \! 6 \\
\end{bmatrix} \! & \! \begin{bmatrix}
 8 \! & \! 6 \! & \! 10 \! & \! 10 \\
 6 \! & \! 18 \! & \! 6 \! & \! 15 \\
 10 \! & \! 6 \! & \! 14 \! & \! 9 \\
 10 \! & \! 15 \! & \! 9 \! & \! 22 \\
\end{bmatrix} \! & \! \begin{bmatrix}
 8 \! & \! -4 \! & \! 8 \! & \! 0 \\
 -4 \! & \! 10 \! & \! -4 \! & \! 0 \\
 8 \! & \! -4 \! & \! 8 \! & \! 0 \\
 0 \! & \! 0 \! & \! 0 \! & \! 0 
\end{bmatrix} \smallskip \\

\end{matrix}
$$
$$
\begin{matrix}

 (6 , 6 ) \, : \! & \!
\begin{bmatrix}
 10 \! & \! 8 \! & \! 2 \! & \! 6 \\
 8 \! & \! 14 \! & \! 0 \! & \! 2 \\
 2 \! & \! 0 \! & \! 5 \! & \! 7 \\
 6 \! & \! 2 \! & \! 7 \! & \! 11 \\
\end{bmatrix} \! & \! \begin{bmatrix}
 11 \! & \! -6 \! & \! 10 \! & \! 9 \\
 -6 \! & \! 10 \! & \! -5 \! & \! -5 \\
 10 \! & \! -5 \! & \! 14 \! & \! 11 \\
 9 \! & \! -5 \! & \! 11 \! & \! 9 \\
\end{bmatrix} \! & \! \begin{bmatrix}
 6 \! & \! 2 \! & \! 6 \! & \! -5 \\
 2 \! & \! 9 \! & \! 2 \! & \! 0 \\
 6 \! & \! 2 \! & \! 6 \! & \! -5 \\
 -5 \! & \! 0 \! & \! -5 \! & \! 5 \\
\end{bmatrix} \! & \! \begin{bmatrix}
 8 \! & \! 6 \! & \! 2 \! & \! -2 \\
 6 \! & \! 9 \! & \! 9 \! & \! 6 \\
 2 \! & \! 9 \! & \! 13 \! & \! 12 \\
 -2 \! & \! 6 \! & \! 12 \! & \! 13 \\
\end{bmatrix} \smallskip \\

 (8 , 8 ) 
\, : \! & \!
\begin{bmatrix}
 5 \! & \! 3 \! & \! -3 \! & \! -4 \\
 3 \! & \! 6 \! & \! -3 \! & \! -2 \\
 -3 \! & \! -3 \! & \! 6 \! & \! 4 \\
 -4 \! & \! -2 \! & \! 4 \! & \! 4 \\
\end{bmatrix} \! & \! \begin{bmatrix}
 19 \! & \! 10 \! & \! 12 \! & \! 17 \\
 10 \! & \! 14 \! & \! 10 \! & \! 7 \\
 12 \! & \! 10 \! & \! 10 \! & \! 11 \\
 17 \! & \! 7 \! & \! 11 \! & \! 17 \\
\end{bmatrix} \! & \! \begin{bmatrix}
 5 \! & \! 1 \! & \! 3 \! & \! -3 \\
 1 \! & \! 5 \! & \! -7 \! & \! -1 \\
 3 \! & \! -7 \! & \! 22 \! & \! 7 \\
 -3 \! & \! -1 \! & \! 7 \! & \! 10 \\
\end{bmatrix} \! & \! \begin{bmatrix}
 1 \! & \! 1 \! & \! 0 \! & \! 2 \\
 1 \! & \! 1 \! & \! 0 \! & \! 2 \\
 0 \! & \! 0 \! & \! 4 \! & \! 4 \\
 2 \! & \! 2 \! & \! 4 \! & \! 8 \\
\end{bmatrix} \smallskip \\

 (10 , 10 )
\, : \! & \!
\begin{bmatrix}
 18 \! & \! 6 \! & \! 6 \! & \! -6 \\
 6 \! & \! 2 \! & \! 2 \! & \! -2 \\
 6 \! & \! 2 \! & \! 2 \! & \! -2 \\
 -6 \! & \! -2 \! & \! -2 \! & \! 4 \\
\end{bmatrix} \! & \! \begin{bmatrix}
 4 \! & \! -6 \! & \! 6 \! & \! 4 \\
 -6 \! & \! 13 \! & \! -9 \! & \! -8 \\
 6 \! & \! -9 \! & \! 9 \! & \! 6 \\
 4 \! & \! -8 \! & \! 6 \! & \! 5 \\
\end{bmatrix} \! & \! \begin{bmatrix}
 1 \! & \! 0 \! & \! -3 \! & \! 0 \\
 0 \! & \! 4 \! & \! 0 \! & \! 6 \\
 -3 \! & \! 0 \! & \! 9 \! & \! 0 \\
 0 \! & \! 6 \! & \! 0 \! & \! 9 \\
\end{bmatrix} \! & \! \begin{bmatrix}
 9 \! & \! -3 \! & \! 0 \! & \! 0 \\
 -3 \! & \! 10 \! & \! 9 \! & \! -6 \\
 0 \! & \! 9 \! & \! 9 \! & \! -6 \\
 0 \! & \! -6 \! & \! -6 \! & \! 4 \\
\end{bmatrix} \smallskip \\
( 2 , 0 )
\, : \! & \!
\begin{bmatrix}
 20 \! & \! 6 \! & \! \!-14 \! & \! -4 \\
 6 \! & \! 18 \! & \! 3 \! & \! -12 \\
 -14 \! & \! 3 \! & \! 17 \! & \! -2 \\
 -4 \! & \!\! -12 \! & \! -2 \! & \! 8 \\
\end{bmatrix} \! & \! \begin{bmatrix}
 54 \! & \!\! -27 \! & \! 16 \! & \! 12 \\
 -27 \! & \! 18 \! & \! -2 \! & \!\!-15 \\
 16 \! & \! -2 \! & \! 20 \! & \!\! -10 \\
 12 \! & \!\! -15 \! & \! -10 \! & \! 21 \\
\end{bmatrix} \! & \! \begin{bmatrix}
 42 \! & \! -8 \! & \! 9 \! & \! -3 \\
 -8 \! & \! 10 \! & \! 5 \! & \! -11 \\
 9 \! & \! 5 \! & \! 29 \! & \! 7 \\
 -3 \! & \! -11 \! & \! 7 \! & \! 29 \\
\end{bmatrix} \! & \! \begin{bmatrix}
 0 \! & \! 9 \! & \! 3 \! & \! -3 \\
 9 \! & \! -9 \! & \! -6 \! & \! 6 \\
 3 \! & \! -6 \! & \! -3 \! & \! 3 \\
 -3 \! & \! 6 \! & \! 3 \! & \! -3 \\
\end{bmatrix} \smallskip \\
( 4 , 2)
\, : \! & \!
\begin{bmatrix}
 9 \! & \! -4 \! & \! 1 \! & \! 1 \\
 -4 \! & \! 5 \! & \! -3 \! & \! -2 \\
 1 \! & \! -3 \! & \! 3 \! & \! 1 \\
 1 \! & \! -2 \! & \! 1 \! & \! 1 \\
\end{bmatrix} \! & \! \begin{bmatrix}
 6 \! & \! 1 \! & \! 3 \! & \! 4 \\
 1 \! & \! 5 \! & \! 5 \! & \! 2 \\
 3 \! & \! 5 \! & \! 6 \! & \! 2 \\
 4 \! & \! 2 \! & \! 2 \! & \! 8 \\
\end{bmatrix} \! & \! \begin{bmatrix}
 8 \! & \! 2 \! & \! -6 \! & \! 4 \\
 2 \! & \! 5 \! & \! 1 \! & \! 3 \\
 -6 \! & \! 1 \! & \! 6 \! & \! -2 \\
 4 \! & \! 3 \! & \! -2 \! & \! 3 \\
\end{bmatrix} \! & \! \begin{bmatrix}
 -4 \! & \! 4 \! & \! -2 \! & \! 2 \\
 4 \! & \! 0 \! & \! 0 \! & \! -2 \\
 -2 \! & \! 0 \! & \! 0 \! & \! 1 \\
 2 \! & \! -2 \! & \! 1 \! & \! -1 \\
\end{bmatrix} \\( 6 , 4) 
\, : \! & \!
\begin{bmatrix}
 6 \! & \! -1 \! & \! 5 \! & \! 5 \\
 -1 \! & \! 2 \! & \! 1 \! & \! -3 \\
 5 \! & \! 1 \! & \! 6 \! & \! 2 \\
 5 \! & \! -3 \! & \! 2 \! & \! 9 \\
\end{bmatrix} \! & \! \begin{bmatrix}
 5 \! & \! -5 \! & \! 5 \! & \! -3 \\
 -5 \! & \! 6 \! & \! -5 \! & \! 5 \\
 5 \! & \! -5 \! & \! 5 \! & \! -3 \\
 -3 \! & \! 5 \! & \! -3 \! & \! 9 \\
\end{bmatrix} \! & \! \begin{bmatrix}
 6 \! & \! -3 \! & \! 5 \! & \! 2 \\
 -3 \! & \! 5 \! & \! -3 \! & \! 2 \\
 5 \! & \! -3 \! & \! 9 \! & \! -4 \\
 2 \! & \! 2 \! & \! -4 \! & \! 9 \\
\end{bmatrix} \! & \! \begin{bmatrix}
 0 \! & \! -2 \! & \! -2 \! & \! 0 \\
 -2 \! & \! 1 \! & \! 2 \! & \! 1 \\
 -2 \! & \! 2 \! & \! 3 \! & \! 1 \\
 0 \! & \! 1 \! & \! 1 \! & \! 0 \\
\end{bmatrix} \smallskip \\
( 8 , 6 )
\, : \! & \!
\begin{bmatrix}
 4 \! & \! 0 \! & \! 4 \! & \! -2 \\
 0 \! & \! 5 \! & \! -2 \! & \! 5 \\
 4 \! & \! -2 \! & \! 8 \! & \! -4 \\
 -2 \! & \! 5 \! & \! -4 \! & \! 6 \\
\end{bmatrix} \! & \! \begin{bmatrix}
 2 \! & \! 3 \! & \! -1 \! & \! -1 \\
 3 \! & \! 6 \! & \! -1 \! & \! -4 \\
 -1 \! & \! -1 \! & \! 6 \! & \! -3 \\
 -1 \! & \! -4 \! & \! -3 \! & \! 6 \\
\end{bmatrix} \! & \! \begin{bmatrix}
 6 \! & \! 2 \! & \! 0 \! & \! 1 \\
 2 \! & \! 8 \! & \! 4 \! & \! -2 \\
 0 \! & \! 4 \! & \! 8 \! & \! -2 \\
 1 \! & \! -2 \! & \! -2 \! & \! 1 \\
\end{bmatrix} \! & \! \begin{bmatrix}
 2 \! & \! -3 \! & \! 0 \! & \! 1 \\
 -3 \! & \! 5 \! & \! 0 \! & \! 0 \\
 0 \! & \! 0 \! & \! 0 \! & \! 0 \\
 1 \! & \! 0 \! & \! 0 \! & \! 5 \\
\end{bmatrix} \smallskip \\
( 10 , 8 )
\, : \! & \!
\begin{bmatrix}
 5 \! & \! -1 \! & \! -1 \! & \! 4 \\
 -1 \! & \! 6 \! & \! -3 \! & \! 5 \\
 -1 \! & \! -3 \! & \! 2 \! & \! -4 \\
 4 \! & \! 5 \! & \! -4 \! & \! 9 \\
\end{bmatrix} \! & \! \begin{bmatrix}
 8 \! & \! 0 \! & \! 0 \! & \! -4 \\
 0 \! & \! 1 \! & \! 0 \! & \! -1 \\
 0 \! & \! 0 \! & \! 2 \! & \! 0 \\
 -4 \! & \! -1 \! & \! 0 \! & \! 3 \\
\end{bmatrix} \! & \! \begin{bmatrix}
 6 \! & \! 5 \! & \! 1 \! & \! -2 \\
 5 \! & \! 9 \! & \! -3 \! & \! -4 \\
 1 \! & \! -3 \! & \! 6 \! & \! 4 \\
 -2 \! & \! -4 \! & \! 4 \! & \! 4 \\
\end{bmatrix} \! & \! \begin{bmatrix}
 8 \! & \! 0 \! & \! 0 \! & \! -4 \\
 0 \! & \! 8 \! & \! 4 \! & \! 4 \\
 0 \! & \! 4 \! & \! 2 \! & \! 2 \\
 -4 \! & \! 4 \! & \! 2 \! & \! 4 \\
\end{bmatrix}  \smallskip \\
( 4 , 0 )
\, : \! & \!
\begin{bmatrix}
 21 \! & \! 10 \! & \! 1 \! & \! -6 \\
 10 \! & \! 10 \! & \! 0 \! & \! -1 \\
 1 \! & \! 0 \! & \! 2 \! & \! -3 \\
 -6 \! & \! -1 \! & \! -3 \! & \! 6 \\
\end{bmatrix} \! & \! \begin{bmatrix}
 0 \! & \! 6 \! & \! -6 \! & \! 2 \\
 6 \! & \! 3 \! & \! 0 \! & \! -4 \\
 -6 \! & \! 0 \! & \! -3 \! & \! 5 \\
 2 \! & \! -4 \! & \! 5 \! & \! -3 \\
\end{bmatrix} \! & \! \begin{bmatrix}
 0 \! & \! 0 \! & \! 0 \! & \! 2 \\
 0 \! & \! 0 \! & \! 0 \! & \! -1 \\
 0 \! & \! 0 \! & \! 0 \! & \! -1 \\
 2 \! & \! -1 \! & \! -1 \! & \! 5 \\
\end{bmatrix} \! & \! \begin{bmatrix}
 0 \! & \! 3 \! & \! -1 \! & \! 1 \\
 3 \! & \! -3 \! & \! 8 \! & \! -5 \\
 -1 \! & \! 8 \! & \! -5 \! & \! 4 \\
 1 \! & \! -5 \! & \! 4 \! & \! -3 \\
\end{bmatrix} \smallskip \\
( 6 , 2 )
\, : \! & \!
\begin{bmatrix}
 7 \! & \! -1 \! & \! 5 \! & \! 2 \\
 -1 \! & \! 5 \! & \! -1 \! & \! 5 \\
 5 \! & \! -1 \! & \! 4 \! & \! 1 \\
 2 \! & \! 5 \! & \! 1 \! & \! 7 \\
\end{bmatrix} \! & \! \begin{bmatrix}
 -1 \! & \! -2 \! & \! 1 \! & \! -2 \\
 -2 \! & \! -3 \! & \! 2 \! & \! -6 \\
 1 \! & \! 2 \! & \! -1 \! & \! 2 \\
 -2 \! & \! -6 \! & \! 2 \! & \! 0 \\
\end{bmatrix} \! & \! \begin{bmatrix}
 4 \! & \! 4 \! & \! 2 \! & \! -2 \\
 4 \! & \! 0 \! & \! 4 \! & \! -2 \\
 2 \! & \! 4 \! & \! 0 \! & \! -1 \\
 -2 \! & \! -2 \! & \! -1 \! & \! 1 \\
\end{bmatrix} \! & \! \begin{bmatrix}
 -1 \! & \! 1 \! & \! 2 \! & \! 1 \\
 1 \! & \! -1 \! & \! -2 \! & \! -1 \\
 2 \! & \! -2 \! & \! -3 \! & \! -1 \\
 1 \! & \! -1 \! & \! -1 \! & \! 0 \\
\end{bmatrix} \smallskip \\

( 8 , 4 )
\, : \! & \!
\begin{bmatrix}
 16 \! & \! -4 \! & \!\! -16 \! & \! 10 \\
 -4 \! & \! 18 \! & \! 0 \! & \! \!-13 \\
 -16 \! & \! 0 \! & \! 20 \! & \!\! -9 \\
 10 \! & \!\! -13 \! & \!\! -9 \! & \! 19 \\
\end{bmatrix} \! & \! \begin{bmatrix}
 0 \! & \! 1 \! & \! -1 \! & \! 0 \\
 1 \! & \! -5 \! & \! 6 \! & \! 1 \\
 -1 \! & \! 6 \! & \! -7 \! & \! -1 \\
 0 \! & \! 1 \! & \! -1 \! & \! 0 \\
\end{bmatrix} \! & \! \begin{bmatrix}
 0 \! & \! \!-16 \! & \! 0 \! & \! -8 \\
 -16 \! & \! 0 \! & \! 16 \! & \! -16 \\
 0 \! & \! 16 \! & \! 0 \! & \! 8 \\
 -8 \! & \!\! -16 \! & \! 8 \! & \! -16 \\
\end{bmatrix} \! & \! \begin{bmatrix}
 7 \! & \! 9 \! & \! 16 \! & \! 3 \\
 9 \! & \! -9 \! & \! -12 \! & \! 9 \\
 16 \! & \! -12 \! & \! -15 \! & \! 15 \\
 3 \! & \! 9 \! & \! 15 \! & \! 0 \\
\end{bmatrix} \smallskip \\

\end{matrix}
$$
$$
\begin{matrix}

( 10 , 6 )
\, : \! & \!
\begin{bmatrix}
 18 \! & \! \! -13 \! & \! 15 \! & \! 1 \\
 -13 \! & \! 22 \! & \! 2 \! & \!\!\! -16 \\
 15 \! & \! 2 \! & \! 30 \! & \!\!\! -20 \\
 1 \! & \!\! -16 \! & \!\! -20 \! & \! 30 \\
\end{bmatrix} \! & \! \begin{bmatrix}
 -15 \! & \! 7 \! & \! 8 \! & \! 5 \\
 7 \! & \!\!\!-3 \! & \!\! -4 \! & \! -3 \\
 8 \! & \!\!\!-4 \! & \!\! -4 \! & \! -2 \\
 5 \! & \!\!\! -3 \! & \!\! -2 \! & \! 0 \\
\end{bmatrix} \! & \! \begin{bmatrix}
 1 \! & \! 0 \! & \! 1 \! & \! -3 \\
 0 \! & \! 0 \! & \! 0 \! & \! 0 \\
 1 \! & \! 0 \! & \! -8 \! & \! -15 \\
 -3 \! & \! 0 \! & \! -15 \! & \! -7 \\
\end{bmatrix} \! & \! \begin{bmatrix}
 -15 \! & \! 0 \! & \! -6 \! & \! 2 \\
 0 \! & \! 15 \! & \! 6 \! & \! 8 \\
 -6 \! & \! 6 \! & \! 0 \! & \! 4 \\
 2 \! & \! 8 \! & \! 4 \! & \! 4 \\
\end{bmatrix} \\

( 6 , 0 )
\, : \! & \!
\begin{bmatrix}
 3 \! & \! 6 \! & \! -4 \! & \! -4 \\
 6 \! & \! 13 \! & \! -5 \! & \! -5 \\
 -4 \! & \! -5 \! & \! 19 \! & \! 20 \\
 -4 \! & \! -5 \! & \! 20 \! & \! 23 \\
\end{bmatrix} \! & \! \begin{bmatrix}
 0 \! & \! -1 \! & \! -3 \! & \! 0 \\
 -1 \! & \! 3 \! & \! 6 \! & \! 0 \\
 -3 \! & \! 6 \! & \! 9 \! & \! 0 \\
 0 \! & \! 0 \! & \! 0 \! & \! 0 \\
\end{bmatrix} \! & \! \begin{bmatrix}
 8 \! & \! 2 \! & \! -2 \! & \! 2 \\
 2 \! & \! -4 \! & \! -2 \! & \! 2 \\
 -2 \! & \! -2 \! & \! 0 \! & \! 0 \\
 2 \! & \! 2 \! & \! 0 \! & \! 0 \\
\end{bmatrix} \!\! & \!\! \begin{bmatrix}
 1 \! & \! -2 \! & \! 1 \! & \! 3 \\
 -2 \! & \! -5 \! & \!\! -11 \! & \!\! -15 \\
 1 \! & \!\! -11 \! & \! -8 \! & \! -6 \\
 3 \! & \!\! -15 \! & \! -6 \! & \! 0 \\
\end{bmatrix} \smallskip \\
( 8,  2 )
\, : \! & \!
\begin{bmatrix}
 3 \! & \! -3 \! & \! 3 \! & \! -1 \\
 -3 \! & \! 4 \! & \! -3 \! & \! 2 \\
 3 \! & \! -3 \! & \! 5 \! & \! 0 \\
 -1 \! & \! 2 \! & \! 0 \! & \! 2 \\
\end{bmatrix} \! & \! \begin{bmatrix}
 -1 \! & \! 1 \! & \! -1 \! & \! -2 \\
 1 \! & \! 0 \! & \! 0 \! & \! 0 \\
 -1 \! & \! 0 \! & \! 0 \! & \! 0 \\
 -2 \! & \! 0 \! & \! 0 \! & \! 0 \\
\end{bmatrix} \! & \! \begin{bmatrix}
 0 \! & \! 0 \! & \! -1 \! & \! -2 \\
 0 \! & \! 0 \! & \! 0 \! & \! 0 \\
 -1 \! & \! 0 \! & \! 1 \! & \! 0 \\
 -2 \! & \! 0 \! & \! 0 \! & \! -4 \\
\end{bmatrix} \! & \! \begin{bmatrix}
 -1 \! & \! 1 \! & \! 1 \! & \! 0 \\
 1 \! & \! 3 \! & \! -1 \! & \! 2 \\
 1 \! & \! -1 \! & \! -1 \! & \! 0 \\
 0 \! & \! 2 \! & \! 0 \! & \! 1 \\
\end{bmatrix} \smallskip  \\
( 10 , 4 )
\, : \! & \!
\begin{bmatrix}
 5 \! & \! -1 \! & \! -3 \! & \! 1 \\
 -1 \! & \! 2 \! & \! 2 \! & \! 0 \\
 -3 \! & \! 2 \! & \! 4 \! & \! -1 \\
 1 \! & \! 0 \! & \! -1 \! & \! 3 \\
\end{bmatrix} \! & \! \begin{bmatrix}
 0 \! & \! 0 \! & \! 0 \! & \! 0 \\
 0 \! & \! -4 \! & \! -4 \! & \! -2 \\
 0 \! & \! -4 \! & \! -4 \! & \! -2 \\
 0 \! & \! -2 \! & \! -2 \! & \! 0 \\
\end{bmatrix} \! & \! \begin{bmatrix}
 0 \! & \! 4 \! & \! -4 \! & \! -6 \\
 4 \! & \! 0 \! & \! 2 \! & \! 1 \\
 -4 \! & \! 2 \! & \! -4 \! & \! -4 \\
 -6 \! & \! 1 \! & \! -4 \! & \! -3 \\
\end{bmatrix} \! & \! \begin{bmatrix}
 -3 \! & \! 0 \! & \! -1 \! & \! -2 \\
 0 \! & \! 0 \! & \! 0 \! & \! 0 \\
 -1 \! & \! 0 \! & \! 0 \! & \! -1 \\
 -2 \! & \! 0 \! & \! -1 \! & \! -1 \\
\end{bmatrix} \smallskip \\
( 8 , 0 )
\, : \! & \!
\begin{bmatrix}
 9 \! & \! 0 \! & \! -7 \! & \! -10 \\
 0 \! & \! 5 \! & \! 0 \! & \! 2 \\
 -7 \! & \! 0 \! & \! 15 \! & \! 5 \\
 -10 \! & \! 2 \! & \! 5 \! & \! 13 \\
\end{bmatrix} \! & \! \begin{bmatrix}
 8 \! & \! 6 \! & \! 5 \! & \! 8 \\
 6 \! & \! -8 \! & \! -5 \! & \! -4 \\
 5 \! & \! -5 \! & \! -3 \! & \! -2 \\
 8 \! & \! -4 \! & \! -2 \! & \! 0 \\
\end{bmatrix} \! & \! \begin{bmatrix}
 8 \! & \! 4 \! & \! 11 \! & \! 4 \\
 4 \! & \! 0 \! & \! 10 \! & \! 0 \\
 11 \! & \! 10 \! & \! 5 \! & \! 10 \\
 4 \! & \! 0 \! & \! 10 \! & \! 0 \\
\end{bmatrix} \! & \! \begin{bmatrix}
 -4 \! & \! -4 \! & \! 2 \! & \! 4 \\
 -4 \! & \! -4 \! & \! 2 \! & \! 4 \\
 2 \! & \! 2 \! & \! 0 \! & \! 0 \\
 4 \! & \! 4 \! & \! 0 \! & \! 0 \\
\end{bmatrix} \smallskip \\
( 10 , 2 )
\, : \! & \!
\begin{bmatrix}
 29 \! & \!\! -22 \! & \! 4 \! & \! -4 \\
 -22 \! & \! 26 \! & \! -7 \! & \! 5 \\
 4 \! & \! -7 \! & \! 25 \! & \! -6 \\
 -4 \! & \! 5 \! & \! -6 \! & \! 5 \\
\end{bmatrix} \!\! & \! \! \begin{bmatrix}
 -1 \! & \! -4 \! & \! -1 \! & \! -4 \\
 -4 \! & \!\! -12 \! & \! -4 \! & \!\! -14 \\
 -1 \! & \! -4 \! & \! -1 \! & \! -4 \\
 -4 \! & \!\! -14 \! & \! -4 \! & \!\! -15 \\
\end{bmatrix} \! & \! \begin{bmatrix}
 -5 \! & \! 9 \! & \! 6 \! & \! 7 \\
 9 \! & \! 8 \! & \! -2 \! & \! 5 \\
 6 \! & \! -2 \! & \! -4 \! & \! -2 \\
 7 \! & \! 5 \! & \! -2 \! & \! 3 \\
\end{bmatrix} \!\! & \!\! \begin{bmatrix}
 -5 \! & \! 16 \! & \! -1 \! & \!\!-10 \\
 16 \! & \!\! -12 \! & \! 20 \! & \! 4 \\
 -1 \! & \! 20 \! & \! 7 \! & \!\! -14 \\
 -10 \! & \! 4 \! & \!\! -14 \! & \! 0 \\
\end{bmatrix}\smallskip \\
( 10 , 0 )
\, : \! & \!
\begin{bmatrix}
 51 \! & \! -34 \! & \! 5 \! & \! 60 \\
 -34 \! & \! 147 \! & \! 30 \! & \!\! -37 \\
 5 \! & \! 30 \! & \! 99 \! & \! 40 \\
 60 \! & \!\! -37 \! & \! 40 \! & \! 135 \\
\end{bmatrix} \!\! & \!\! \begin{bmatrix}
 15 \! & \! 97 \! & \! 64 \! & \! 36 \\
 97 \! & \! -13 \! & \! -50 \! & \! 76 \\
 64 \! & \! -50 \! & \! -63 \! & \! 40 \\
 36 \! & \! 76 \! & \! 40 \! & \! 48 \\
\end{bmatrix} \!\! & \!\! \begin{bmatrix}
 -27 \! & \! 45 \! & \!\! -27 \! & \! 51 \\
 45 \! & \! 0 \! & \!\! -30 \! & \! 10 \\
 -27 \! & \!\! -30 \! & \! 48 \! & \!\! -44 \\
 51 \! & \! 10 \! & \!\! -44 \! & \! 24 \\
\end{bmatrix} \!\! & \!\! \begin{bmatrix}
 -60 \! & \! 30 \! & \! 10 \! & \!\! -52 \\
 30 \! & \! 45 \! & \!\! -55 \! & \! -2 \\
 10 \! & \! \! -55 \! & \! 40 \! & \! 32 \\
 -52 \! & \! -2 \! & \! 32 \! & \!\! -32 \\
\end{bmatrix}
\end{matrix}
$$
%\end{small}
}

\noindent
This list of $20$ representatives completes the proof of the 
 if-direction in Theorem \ref{thm:DI}.
\end{proof}

\section{Ramification and Cayley's Theorem} \label{sec:nodes}

The classical approach to nodal quartics in $3$-space is via their double cover of the plane when 
projected from a node. Relevant references include Cayley \cite{Ca1},  Coble \cite{Co}, and Jessop \cite{J}.
See also Cossec \cite{Cos} and Dolgachev \cite{Do} for modern proofs of some of these results.
We here develop this approach over $\R$, for quartic spectrahedra.
The goal is to prove Theorem~\ref{thm:cayley}.

We note that  quartic symmetroids  and their ramification
over two cubics play a crucial role in
 Huh's recent counterexample
to the geometric Chevalley-Warning conjecture~\cite{Huh}.

Let $X=V(f)\subset \C\PP^3$ be a quartic surface and $p$ a node on $X$. 
 The projection from $p$ defines a rational 2:1 map to the plane $\C\PP^2$.
  It extends  to a morphism on the  blow-up
$\tilde X\to X$ with exceptional curve $E_p$ lying over $p$.   We denote 
this extended double cover by
\[\pi_p= \pi_p(\C):\tilde X\to \C\PP^2. \] 
 If $p$ is a real point, then $\pi_p$ maps the real part of $X$ to the real projective plane:
\[
\pi_p(\R):\tilde X_{\R}\to \R\PP^2.
\]
Assume $p=(1:0:0:0)$. Then
$f\,=\,ax_0^2+bx_0+c$ with $ a,b,c\in \C[x_1,x_2,x_3]$.
The fiber of $\pi_p$ is given by the roots of $f$,
   and the {\em ramification locus}, where $f$ has a double root, is the sextic curve  $R_p = V(b^2-4ac)$.
The lines in $X$ through $p$ and the singular points 
of $ X \backslash p$  map to the singular points of $R_p$.
Each node of $R_p$ is the image of
 a node on $X \backslash p$
or the image of a  line in $X$ through $p$.

  The real projection $\pi_p(\R)$ is ramified in the real points of the 
  sextic curve $R_p$.
  The quadratic equation $a=0$ defines in $\C\PP^3$ the 
 {\em tangent cone} of $X$ at $p$. In $\C\PP^2$ it defines a
 conic section  $C_p$.  Since the
  equation $b^2-4ac$ of $R_p$  is a square modulo $a$, the
  sextic $R_p$ is {\em totally tangent} to the conic $C_p$, i.e. tangent at every point of intersection. 

% under $\pi_p$ of the exceptional curve $E_p$ on $\tilde X$
Suppose now that $f(x) = {\rm det}(A(x))$, with $A(x)$  as in (\ref{eq:AxAx}).
  A node $p$ on the symmetroid $X$ has rank $2$ or $3$.  
  If $X$ is transversal then  every node has rank $2$.
  A real node $p$ on a very real symmetroid $X$ is an
  {\em $(m,n)$-node} if the matrix $A(p)$ has $m$ positive and $n$ negative eigenvalues.
  Because we are working projectively, we can always take $m\geq n$. 
  Thus there are two types of rank-2 nodes: $(2,0)$-nodes that lie on the spectrahedron $S(f)$
  and $(1,1)$-nodes that do not. 

Cayley discovered that 
 if the node $p$ has rank $2$, then  the ramification locus is the union of two cubic curves
 and vice versa (cf.~\cite{Ca1, Cos, J}). 
    We now prove our extension of his result.

 \begin{proof}[Proof of Theorem~\ref{thm:cayley}]
 ($\Rightarrow$) Suppose that $p$ is a rank-$2$ node on the symmetroid  $X=V(f)$.
After a change of coordinates, we can take $p=(1:0:0:0)$. 
By conjugating with an appropriate matrix (over $\C$), we can write 
 $f$ as the determinant of the symmetric matrix 
\begin{equation}
\label{eq:symmatrix1}
\begin{bmatrix}
m_{11}&x_0+m_{12}&m_{13}&m_{14}\\
x_0+m_{12}& m_{22}&m_{23}&m_{24}\\
m_{13}&m_{23}&m_{33}&m_{34}\\
m_{14}&m_{24}&m_{34}&m_{44}\\
\end{bmatrix}
\end{equation}
where each $m_{jk}$ is a linear form in $ x_1,x_2,x_3$.
This determinant equals
\begin{equation}\label{eq:RamDet}
f \quad = \quad
- q \cdot x_0^2
\,+\, 2  F_{12} \cdot x_0 \,+ \, \Delta,
\;\;\;\text{ where }\;\;\;
q = \left |\begin{matrix}
m_{33}&m_{34}\\
m_{34}&m_{44}\\
\end{matrix} \right |.
\end{equation}
Here $\Delta$ is the determinant of the matrix $(m_{jk})$, and 
the cubic $F_{jk}$ denotes the cofactor of $m_{jk}$ in this matrix.
The ramification locus $R_p$ is the curve in $\C\PP^2$  defined by the sextic equation
\[
0 \;\;=\;\; (2F_{12})^2 -4(- q) \Delta \;\; = \;\; 4 (F_{12}^2+q \Delta) .
\]
Using the determinantal formula $\,q \Delta  \, = \, F_{11}F_{22}-F_{12}^2  $,
this sextic factors into two cubics:
\[
F_{12}^2+q \Delta\;\; = \;\; F_{11} \cdot F_{22}.
\]
Any singular point of $R_p$ gives a singular point of $X$, so if $X$ is transversal, $F_{11}$ and $F_{22}$ are both smooth.

(\textit{1.2.3}$\; \Rightarrow$) 
Suppose that the symmetroid $f$ is very real and $p$ is a $(1,1)$-node.
 There exists a real change of coordinates taking $p$ to $(1:0:0:0)$.
 The quartic $f$ can be written as the determinant \eqref{eq:symmatrix1} where the
$m_{jk}$ are linear forms in $\R[x_1, x_2, x_3]$.  The above argument shows that
the two ramification cubics $F_{11}$ and $F_{22}$ are real. \smallskip

(\textit{1.2.2}$\; \Rightarrow$)
Suppose instead that $f$ is very real and $p$ is a $(2,0)$-node.
We again make a real change of coordinates so that $p=(1:0:0:0)$.
Then $f$ has a determinantal representation 
$$
{\small
\left|\begin{matrix}
\frac{1}{2}x_0+\ell_{11} \!&\!\ell_{12} \!&\!\ell_{13} \!&\!\ell_{14}\\
\ell_{12} \!&\!  \frac{1}{2}x_0+\ell_{22} \!&\!\ell_{23} \!&\!\ell_{24}\\
\ell_{13} \!&\!\ell_{23} \!&\!\ell_{33} \!&\!\ell_{34}\\
\ell_{14} \!&\!\ell_{24} \!&\!\ell_{34} \!&\!\ell_{44}\\
\end{matrix}\right|
\;=\; -\frac{1}{4}
\left|\begin{matrix}
\ell_{11}-\ell_{22}+ 2i \ell_{12} \!&\! \ell_{11}+\ell_{22}+ x_0  \!&\! \ell_{13}+i \ell_{23}   \!&\! \ell_{14}+i \ell_{24}  \\
 \ell_{11}+\ell_{22}+ x_0  \!&\! \ell_{11}-\ell_{22}- 2i \ell_{12}  \!&\! \ell_{13}-i \ell_{23}  \!&\! \ell_{14}-i \ell_{24} \\
 \ell_{13}+i \ell_{23}  \!&\! \ell_{13}-i \ell_{23}   \!&\! \ell_{33}  \!&\! \ell_{34} \\
 \ell_{14}+i \ell_{24} \!&\! \ell_{14}-i \ell_{24}  \!&\! \ell_{34}  \!&\! \ell_{44} \\
\end{matrix}\right|
}
$$
with real linear forms $\ell_{jk} \in \R[x_1,x_2,x_3]$.
The right matrix  is obtained by conjugating the left with an 
appropriate complex matrix to return it to the form \eqref{eq:symmatrix1}.  
The ramification cubics are still given by the cofactors  $F_{11}$, $F_{22}$
of the right-hand matrix, which are complex~conjugates. \vskip 0.1cm

(\textit{1.2.1}$\; \Rightarrow$) 
This is immediate from the statements
(\textit{1.2.2}$\; \Rightarrow$) and
(\textit{1.2.3}$\; \Rightarrow$).
\smallskip

To finish the ``forward'' directions of Theorem~\ref{thm:cayley}, note that 
in both of these real cases, the quadratic $q$ defining $C_p$ has a real symmetric determinantal 
representation \eqref{eq:RamDet}.  This implies that $C_p$ must have a real point,  since
a positive quadratic form is not very real.  

\smallskip

($\Leftarrow$) Suppose first that $X=V(f)$ is a quartic surface with a node
at $p=(1:0:0:0)$ and that the ramification 
locus factors into two smooth cubic curves $R_1=V(F_{11})$ and $R_2=V(F_{22})$. We suppose furthermore that $R_1$ and $R_2$ intersect transversally in $9$ points and that the tangent conic $C_p=V(q)$ is smooth. Then $X$ has precisely $10$ nodes,
  and the quartic has the form $f= -q x_0^2 + 2 g x_0 + \Delta$. The assumption that $X$ has no line through $p$ gives that $V(q,g,\Delta)=\emptyset$.
We shall verify that $X$ is a symmetroid by
constructing a symmetric determinantal representation, up to sign, of $f$. By a continuity argument, 
this in fact shows that any nodal quartic with ramification locus consisting of two cubics is a symmetroid.

We first 
construct a symmetric determinantal representation for the ternary quartic $\Delta$.
This will be done using Dixon's method \cite{Dixon}.
Up to rescaling of $F_{11} $, we have
\[ F_{11} F_{22}\;=\;  ( (2g)^2 - (-q)\Delta  )/4  \;=\; g^2+q\Delta
\qquad \quad \hbox{in} \,\, \C[x_1,x_2,x_3].
\]
Since $q\Delta$ is a square modulo $F_{11}$, there are subschemes $Z_q$ and $Z_\Delta$  in $V(F_{11})$, satisfying
\[ 2Z_q\;=\;V(q,F_{11}) \;\;\;\;\text{ and }\;\;\;\;2Z_\Delta\;=\;V(\Delta,F_{11}),\]
 of length $3$ and $6$ (resp.).
Since $V(q,g,\Delta)=\emptyset$, these are disjoint and $Z_q \cup Z_\Delta =   V(g,F_{11})$.
  
 We claim that $Z_\Delta$ is not contained in a conic.  If it were, this conic together with a line $L$ through any length 
 $2$ subscheme of $Z_q$ is a cubic curve through a subscheme of length $8$ in $V(g,F_{11})$.  By the 
 Cayley-Bacharach Theorem, this cubic contains all of $V(g,F_{11})$.
   Hence   $L$ contains $Z_q$.  But  this is impossible since
   $Z_q$ is a length $3$ subscheme of the smooth conic~$V(q)$.

Consider now the space of cubics in $\C[x_1, x_2, x_3]_3$ vanishing on   $Z_\Delta$.  This space is $4$-dimensional by the previous paragraph.
We extend $F_{11}$ and $F_{12} := g$ to a basis $\{F_{11}, F_{12}, F_{13}, F_{14}\}$ of 
that linear space. 
For each $2\leq j\leq k \leq 4$, the polynomial $F_{1j}F_{1k}$ vanishes to order two on $Z_\Delta$
and hence lies in $\langle F_{11}, \Delta \rangle$.
We can find a cubic $F_{jk} \in \C[x_1, x_2, x_3]_3$ 
and a conic $q_{jk} \in \C[x_1, x_2, x_3]_2$ satisfying
\[ F_{1j}F_{1k}\, =\, F_{11}F_{jk} + q_{jk}\Delta.  \]
Let $F$ be the symmetric matrix of cubics $(F_{jk})$.    By construction, the $2\times 2$ minors of $F$ lie
in $\langle \Delta \rangle$. 
It follows that $\Delta^2$ divides the $3\times 3$ minors of $F$, and  $\Delta^3$ divides 
the determinant of $F$.  Since both $\Delta^3$ and $\det(F)$ have degree $12$, 
we have $\det(F) = c\cdot \Delta^3$ for some constant
$c\in \C$. 
The fact that $Z_\Delta$ is not contained in a conic
ensures that $c\not = 0$, so we get a determinantal 
representation for $X$. This was proved by
Dixon \cite{Dixon}. It is also featured in \cite[Theorem 2.3]{PSV}.

Now consider the adjoint matrix $F^{\rm adj}$  whose entries are the cofactors of $F$. 
These entries have degree nine and are divisible by $\Delta^2$. 
Dividing by $\Delta^2$, we obtain the $4\times 4$ matrix 
\[M=  (1/\Delta^2) \cdot F^{\rm adj}\] 
with linear entries $m_{jk}$. 
A calculation (or the proof of \cite[Theorem~4.6]{PV}) shows that 
\[\det(M) = c^3 \cdot  \Delta,  \;\;\;\;\;\;\;\; M^{\rm adj} = c^2 \cdot F, \;\;\;\;\;\;\;\;  \text{and}
\;\;\;\;\;\;\;\;  \det(M_{34}) = c\cdot q \]
where $M_{34}$ is the $2\times 2$ submatrix of $M$
with rows and columns $\{3,4\}$.
 In particular, $(M^{\rm adj})_{11} = c^2 \cdot F_{11}$,  
 $(M^{\rm adj})_{22} = c^2 \cdot  F_{22}$, and 
 $(M^{\rm adj})_{12} = c^2 \cdot g$.  Now consider the matrix
\begin{equation} \label{eq:DetRep}
A(x) \;\;\;=\;\;\; 
\begin{bmatrix}
m_{11}&c x_0+m_{12}&m_{13}& m_{14}\\
c x_0 +m_{12}&m_{22}&m_{22}&m_{24}\\
m_{13}&m_{23}&m_{33}&m_{34}\\
m_{14}&m_{24}&m_{34}&m_{44}\\
\end{bmatrix} .
\end{equation}
The determinant of $A(x)$ equals
\[-\det(M_{34})\cdot (cx_0)^2 + 2(M^{\rm adj})_{12}\cdot (cx_0) + \det(M)
\;\;= \;\; c^3 \cdot ( -qx_0^2 + 2 g x_0 + \Delta)
\;\;= \;\; c^3 \cdot f,
\]
Since $c \not= 0$, this determinantal representation of $c^3f$ shows
that $V(f)$ is a symmetroid. This shows the first part of Theorem~\ref{thm:cayley}.

(\textit{1.3.3}$\; \Leftarrow$)
Suppose that $f$ is real. Then  $F_{12}$ and $\Delta$ are real. If, in addition,
the ramification cubic $F_{11}$ is real, then the ideal $\langle F_{11}, \Delta \rangle$ 
is defined over $\R$.  Thus all cubics $F_{jk}$ can be taken in $\R[x_1, x_2, x_3]$.
It follows that $c\in \R$, and \eqref{eq:DetRep} gives a real determinantal representation, up to sign, of $f$.
Furthermore, we see that  $A(p) = A((1{:}0{:}0{:}0))$ is not  semidefinite.

(\textit{1.3.2}$\; \Leftarrow$)
Suppose that the ramification cubics $F_{11} , F_{22}$ are complex conjugates. 
This case is more delicate. Up to \emph{real} rescaling of $ F_{11}$ and $F_{22}$, we only know that
$\pm F_{11} F_{22}$ equals the discriminant $F_{12}^2+q\Delta$. 
Suppose $- F_{11} F_{22} = F_{12}^2+q\Delta$.
Then $- q \Delta$ is a sum of squares:
\[ -q  \Delta \;\; = \;\; F_{12}^2 + F_{11} F_{22}   \;\;=\;\; F_{12}^2 + 
({\rm Re}(F_{11}))^2 +  ({\rm Im}(F_{11}))^2.    \]
This implies that $-q$ and $\Delta$ are 
both nonpositive or both nonnegative. 
But this is not possible since the conic $C_p = V(q)$ has real points.
Therefore we can assume  $F_{11} F_{22}=F_{12}^2+q\Delta$.

In the algorithm above,
we can then chose $F_{1j}$ and $F_{2j}$ 
to be complex conjugates,
and we can choose $F_{33},F_{34},F_{44}$ to be real.
  Taking the adjoint of $F$ and dividing by $\Delta^2$ results 
in a matrix $M = (m_{jk})$ with the same real structure, namely that conjugation induces the 
involution $(1 \leftrightarrow 2)$ in the indices of the entries $m_{jk}$. 
Conjugating \eqref{eq:DetRep} by an appropriate complex matrix 
produces the following real determinantal representation of $f$:
\begin{equation}\label{eq:DetRep2}
{\small
 \begin{bmatrix}
2 c x_0+ m_{11}+m_{22}+2 m_{12} \!&\! i (m_{11}- m_{22})  \!&\!  m_{13}+m_{23}  \!&\!  m_{14}+m_{24} \\
 i (m_{11}- m_{22})  \!&\! 2 c x_0  -m_{11}-m_{22}+2 m_{12} \!&\!  i( m_{13}- m_{23})  \!&\!  i (m_{14}- m_{24}) \\
 m_{13}+m_{23}  \!&\!  i (m_{13}-m_{23} ) \!&\!  m_{33}  \!&\!  m_{34} \\
 m_{14}+m_{24}  \!&\!  i (m_{14}- m_{24} ) \!&\!  m_{34}  \!&\!  m_{44} \\
 \end{bmatrix} .
 }
\end{equation}
 This symmetric matrix is real, and it is semidefinite at $p=(1:0:0:0)$. \vskip 0.1cm

(\textit{1.3.1}$\; \Leftarrow$) 
This is immediate from the statements
(\textit{1.3.2}$\; \Leftarrow$) and
(\textit{1.3.3}$\; \Leftarrow$). \vskip 0.1cm

\noindent This concludes our proof of 
Theorem \ref{thm:cayley}.
We shall see a refinement in Proposition \ref{interlacing}.
\end{proof}

\begin{corollary} \label{cor:ninepoints}
In the situation of Theorem \ref{thm:cayley},
the nine intersection points of
the cubics $R_1$ and $R_2$
are precisely the images of the rank-$2$ nodes
on the symmetroid other than~$p$.
\end{corollary}

\begin{proof}
A transversal symmetroid has exactly $10$ nodes, all of rank $2$.
The nodes distinct from $p$ are all mapped to singular points of the sextic $R_p$. These are
the intersection points of the two components $R_1$ and $R_2$.
 Since any symmetroid is a limit of transversal symmetroids, the rank-$2$ nodes  are always mapped by $\pi_p(\C)$ to $R_1\cap R_2$, provided $p$ itself is a rank-$2$ node.
\end{proof}

 There exist nodal quartic symmetroids  with more than ten nodes. 
 For instance, the Kummer symmetroids  in Section \ref{sec:treats} have $16$ nodes, six of which have rank $3$.
 Here each of the cubics  $R_1$ and $R_2$
 factors into three lines, so $R_1 \cup R_2$ has
 $15$ singular points, but $|R_1 \cap R_2| = 9$.

For a transversal quartic symmetroid, the
representation (\ref{eq:AxAx}) is unique
up to conjugation; for a proof see e.g.~\cite[Proposition 11]{BHORS}.
However, if the symmetroid is nodal, but not transversal,
then several  inequivalent determinantal representations can exist.
Each has ten rank-$2$ nodes, and other nodes of rank $3$.
For the Kummer symmetroid, the six rank-$3$ nodes lie
on a plane in $\C \PP^3$, and there are $16$ such planes.
This gives $16$ inequivalent
determinantal representations. Furthermore,
every node has rank $2$ in
{\em some} determinantal representation.
There exist other symmetroids, with exactly $11$ nodes,
where one node $p$ has
  rank $3$ in {\em every} symmetric determinantal representation.
If we project from such a node $p$,
then the ramification curve $R_p$ is a sextic that 
 cannot be decomposed into two cubic curves \cite[I.9]{J}.

The ($\Leftarrow$) direction in the proof above constitutes an algorithm
for computing the unique determinantal representation of a 
transversal symmetroid from its quartic
 $f$. We illustrate this algorithm
with an example. 
An alternative algorithm based on syzygies
is given in
\cite[\S 5]{BHORS}.\\

%\vskip -0.5cm
\begin{figure}[h]
\begin{center}
\vskip -0.4cm
\includegraphics[height = 2.5in]{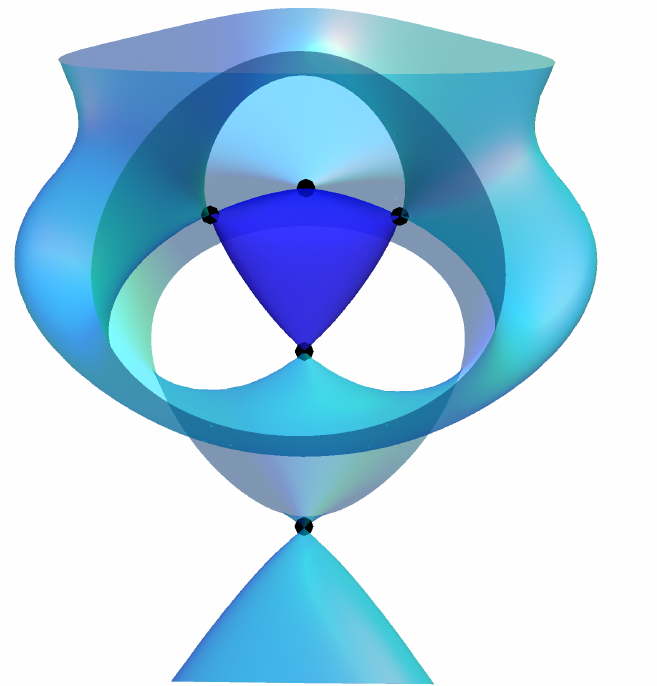}  \quad
\includegraphics[height = 2.45in]{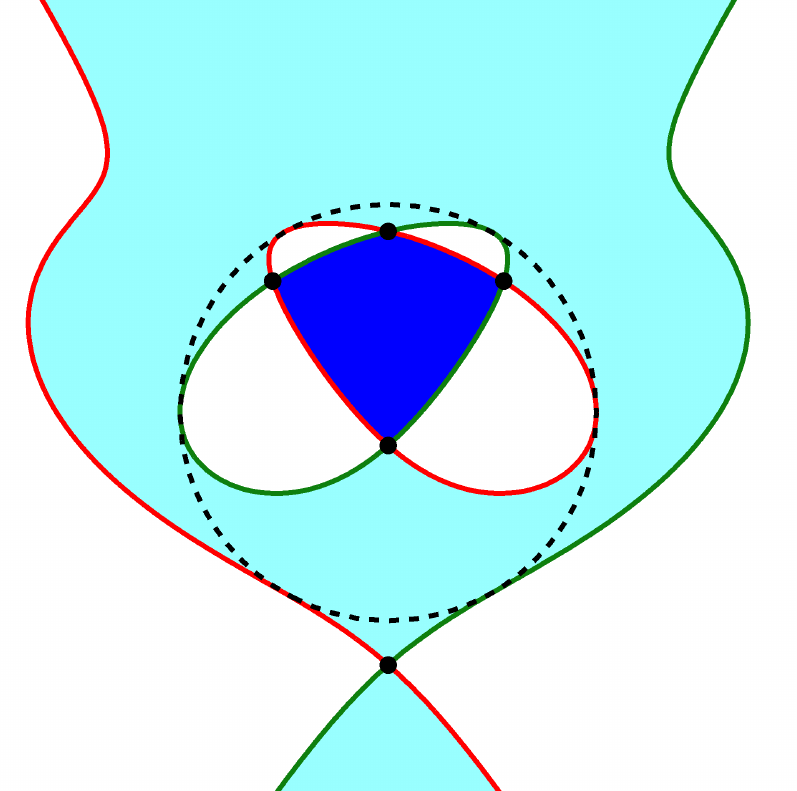}
\end{center}
\vskip -0.5cm
\caption{ \label{fig:construction}
The quartic spectrahedron in Example~\ref{ex:construction} and its projection from a node.}
\end{figure}
\vskip -0.5cm

\begin{example}\label{ex:construction}
The  surface $X = V(f)$ shown in Figure~\ref{fig:construction} is given by the quartic polynomial
\[f = -q x_0^2+2 g x_0 + \Delta, \;\;\;\;\; \text{where} \;\;\;\;\; q = x_1^2 - x_2^2 - x_3^2,  \;\;\; g =  8x_3^3-6x_1^2x_3,  \;\;\; \text{and}\]
 \[\Delta =  x_1^4 - 41x_1^2x_2^2 + 16x_2^4 + 12x_1^3x_3 -  36x_1x_2^2x_3 - 5x_1^2x_3^2 + 44x_2^2x_3^2 -  36x_1x_3^3 + 28x_3^4.\] 
The ramification locus from the node $p = (1:0:0:0)$ is the product $F_{11}F_{22}$ of the cubics
\begin{align*} 
F_{11} \;\;&=\;\; x_1^3+6x_1^2x_2-3x_1x_2^2-4x_2^3+6x_1^2x_3-6x_2^2x_3-3x_1x_3^2-12x_2x_3^2-6x_3^3, \text{ and}   \\
F_{22} \;\;&=\;\; x_1^3-6x_1^2x_2-3x_1x_2^2+4x_2^3+6x_1^2x_3-6x_2^2x_3-3x_1x_3^2+12x_2x_3^2-6x_3^3.
\end{align*}
We check that $F_{11}F_{22}- g^2 = q\Delta$.  The ideal $\sqrt{\langle F_{11},\Delta \rangle }$ contains a 4-dimensional space of
cubics. We extend $F_{11}$ and $F_{12} =g $ to a basis $\{ F_{11}, F_{12},F_{13},F_{14} \}$ of this linear space  
where 
\begin{align*} 
F_{13} \;\;&=\;\;  3x_1x_2x_3+12x_2^2x_3-9x_1x_3^2+18x_2x_3^2+10x_3^3, \text{ and }\\
F_{14} \;\;&=\;\; x_1^2x_2+7x_1x_2^2+4x_2^3-18x_2^2x_3+23x_1x_3^2-32x_2x_3^2-26x_3^3.
\end{align*}
Then, for example, the product $F_{12}F_{13}$ lies in the ideal $\langle F_{11}, \Delta \rangle$, and we can write it as
 {\footnotesize \[ 
\frac{1}{4}(-x_1^3+45x_1x_2^2-28x_2^3-6x_1^2x_3-36x_1x_2x_3+42x_2^2x_3+33x_1x_3^2-36x_2x_3^2-30x_3^3)F_{11} \;+\; \frac{1}{4}(x_1^2+6x_1x_2-7x_2^2+5x_3^2)\Delta. 
\]}We set $F_{23}$ to be the coefficient of $F_{11}$. Similarly, we find $F_{jk}$ for $2\leq j \leq k \leq 4$. 
The determinant of the matrix of cubics $F = (F_{jk})_{jk}$ is $(9/16)\cdot \Delta^3$. 
Taking the adjoint of $F$  and dividing by $\Delta^2$ gives a matrix whose determinant
is $(9/16)^3\cdot \Delta$.  Adding the appropriate matrix $x_0A_0$, as in \eqref{eq:DetRep},
gives a determinantal representation of $(9/16)^3\cdot f$:
\[ \small{
 \frac{1}{32}\begin{bmatrix}
  18 x_1-108 x_2+108 x_3 & 18 x_0+114 x_3 & 24 x_3-72 x_2 & -36 x_2 \\
 18 x_0+114 x_3 & 19 x_1+121 x_2+108 x_3 & 4 x_1+28 x_2+60 x_3 & 30 x_3 \\
 24 x_3-72 x_2 & 4 x_1+28 x_2+60 x_3 & 160 x_1-128 x_2-96 x_3 & 72 x_1-72 x_2-24 x_3 \\
 -36 x_2 & 30 x_3 & 72 x_1-72 x_2-24 x_3 & 36 x_1-36 x_2 
 \end{bmatrix}.
}\]
\end{example}

\section{The view from a node}
\label{sec:view}

We now embark on proving the "only if" direction of Theorem 1.1. In particular, we will show that the number of real nodes on the spectrahedron is even. For this we need to take a closer look at the
projection of a quartic spectrahedron $S$ from
a real node $p$ on its symmetroid $X$. 

To do this, we rely heavily on the notation and proof of Theorem~\ref{thm:cayley}, in particular
the conic $C_p$ and ramification cubics $R_1$ and $R_2$. In the notation of \eqref{eq:RamDet},
the tangent cone of $X$ at $p$ is the non-empty quadratic cone $V(q)$ in $\R \PP^3$.
  It is defined by the determinant of the symmetric $2\times 2$ matrix $\left |\begin{matrix}
m_{33}&m_{34}\\
m_{34}&m_{44}\\
\end{matrix} \right |$ of real linear forms. 
A point $p^{\prime} \in \R \PP^3$ is {\em inside} that cone if 
this quadratic polynomial is nonnegative at $p'$ and {\em outside} the cone otherwise. 
The image of a node $\pi_p(p^{\prime})$ lies in the convex hull of $C_p$ if and only if 
$p'$ lies in the tangent cone of $p$. 
  
Recall that a smooth real cubic curve in the projective plane has a unique {\em pseudo-line}, 
namely a real connected component that is non-contractible (as a loop in $\mathbb R\PP^2$). 
The other connected component, if non-empty, is the {\em oval} of the cubic. Removing an oval disconnects $\mathbb R\PP^2$, 
whereas removing a pseudo-line does not. 
A spectrahedral cubic plane curve has an oval, which is precisely the boundary of the spectrahedron.

\begin{proposition}\label{prop:inside}
Let $S$ be a transversal quartic spectrahedron, fix
  a rank-$2$ node $p$ on its symmetroid~$X$, and let
  $C_p$ be the image of the tangent cone under  $\pi_p(\R)$.
 Suppose no line through $p$ lies in $X$.
 Then each point in $S$ lies inside the tangent cone of $X$ at $p$.
Furthermore:
\begin{itemize}
\item[(a)]
If  $p \in S$, then the image of $S$
   in $\R \PP^2$ is the convex hull of the conic $C_p$.
    Every other node in $\partial S$ is inside the cone and every node 
    in $X \backslash S$
    is outside the cone.
    \item[(b)] If  $p\in X \backslash S$, then  the
    image of $S$ is  the  intersection of the ovals of the ramification cubics, both contained in the interior of the conic $C_p$.
Every node in $\partial S$ is mapped to the boundary of this convex set. 
Each node in $X \backslash S$ is mapped to a point outside $C_p$.
\end{itemize}
\end{proposition}

\begin{proof}[Proof of Proposition \ref{prop:inside}]
Let $p$ be a node of $X$. After a change of coordinates, we can take $p=(1{:}0{:}0{:}0)$, and 
$X$ has a real determinantal representation $A(x)$ given in \eqref{eq:DetRep} or \eqref{eq:DetRep2},
where $m_{jk}$ are linear forms in $x_1, x_2, x_3$.
In either case, the conic $C_p$ is defined by the principal minor $q = m_{33}m_{44}-m_{34}^2$.
For any point $(e_0, e)$ in the spectrahedron $S$, the matrix $A(e_0,e)$ is positive definite 
and thus its principal minor $q$ is positive at $(e_0,e)$.
So, if there exists $e_0 \in \R$ for which $(e_0,e)\in S$, then $q(e)\geq 0$. 
This means that the image of the spectrahedron $S$ is a convex set contained in the convex hull of the conic $C_p$, 
namely $\{e \in \R \PP^2 \;:\; q(e) \geq 0\}$. 

If $p\in S$, then $X$ has a real representation \eqref{eq:DetRep2}. If $q(e) \geq  0$, 
then the $2\times 2$ matrix defining $q$ is semidefinite. Taking $e_0$ sufficiently large makes the 
matrix \eqref{eq:DetRep2}
definite at the point $(e_0,e)$. Thus the projection of $S$ is exactly the convex hull of $C_p$. 
Taking the minimum such $e_0$ shows that the projection of the boundary of the spectrahedron 
is the same set. In particular, the images of all the nodes in $\partial S$ lie inside this convex hull.  
Furthermore, if the image of a node $p^{\prime}$ lies in the convex hull of $C_p$, this implies that 
some point on the line between $p$ and $p^{\prime}$ lies on the boundary $\partial S \backslash \{p\}$. 
However, because these are both nodes, $p$ and $p^{\prime}$ are the only 
intersection points of this line with the surface $X$.  Hence $p^{\prime} \in \partial S$. 

If $p \notin S$, then the surface $X$ has a real determinantal representation $A(x)$ \eqref{eq:DetRep}. 
The ramification cubics are given by diagonal minors of $A(x)$. This shows that these two cubic plane curves 
have ovals bounding planar spectrahedra and that the image $\pi_p(S)$ is contained in the convex hull 
of each oval. Furthermore, both of the $3\times 3$ matrices defining these cubics 
have a diagonal submatrix whose determinant 
is the conic defining $C_p$.  So both of the cubic ovals are contained in the convex hull of the conic $C_p$. 
For any node $p^{\prime}\in S$, the $3\times 3$ minors of $A(p^{\prime})$ vanish, in particular 
the minors defining the ramification cubics. So the projection $\pi_p(p^{\prime})$ lies on both cubic ovals, 
which in turn lie in the convex hull of $C_p$. 
Finally, if $p^{\prime}$ is a node in $X\backslash S$, then the matrix $A(p^{\prime})$ 
is not semidefinite.  It follows that every diagonal $2\times 2$ minor of the rank-two matrix $A(p')$ is non-positive. 
In particular, the $2\times 2$ minor defining the tangent cone is non-positive, 
meaning that the image $\pi_p(p^{\prime})$ does not lie in the interior of the convex hull of $C_p$. 
Because $X$ does not contain the line joining $p$ and $p'$, 
the image of $\pi_p(p^{\prime})$ cannot lie on the conic $C_p$, so $\pi_p(p^{\prime})$ must lie strictly outside of its convex hull. 
\end{proof} 

We expand on the relationship between tangent cones and ramification cubics in 
Proposition~\ref{interlacing}. 
Now we are prepared to complete our new proof of
the Degtyarev-Itenberg Theorem.

\begin{proof}[Proof of Theorem \ref{thm:DI}]
Let $X = V(f)$ be a very real symmetroid whose spectrahedron $S$ is non-empty.
Write $\mathcal N$ for the set of real nodes in $S$.
We claim that $|\mathcal{N}|$ is even. If all  real nodes of $X$ are in $S$,
we are done, since non-real nodes come in conjugate pairs. We may thus assume that 
$X $ has at least one real $(1,1)$-node $p$.
 It suffices to show that $|\pi_p(\mathcal N)|$ is~even.

By Theorem~\ref{thm:cayley}, the ramification locus of $\pi$ is the union 
$R_1 \cup R_2$ of two real cubics.
 By perturbing the symmetric matrices $A_i$ in  $f(x) = {\rm det}(A(x))$,
we may assume that the cubics are smooth and intersect transversely.  In particular, we may assume that $X$ belongs to the open set of  transversal symmetroids that contains no lines.
By Proposition \ref{prop:inside}, the
image $\pi(S)$ of the spectrahedron is bounded by the ovals of the cubics $R_1$ and $R_2$. 

Furthermore, $\pi(\mathcal{N})$ is the set of intersection points of the ovals of $R_1$ and $R_2$.
 But these intersect an even number of times, by the Jordan curve theorem.
Thus  $|\mathcal{N}|$ is even.

The following lemma now implies
Theorem~\ref{thm:DI}. %We present two proofs of Lemma~\ref{lem:atleastone}.
\end{proof}

\def\RR{\mathbb R}

   \def\HH{\mathbb H}
\renewcommand\i{\mathbf i}
\renewcommand\j{\mathbf j}
\renewcommand\k{\mathbf k}

\begin{lemma} \label{lem:atleastone}
Every quartic symmetroid with non-empty spectrahedral region has a  real node.
\end{lemma}

\begin{proof}
After conjugating by an appropriate matrix, we may assume that the identity matrix is
in the spectrahedron $S$, i.e.,
$f(x)=\det A(x)$ where $A(x)=x_0A_0+x_1A_1+x_2A_2+x_3I$.
By a result of Friedland {\it et al.} \cite[Theorem B]{FRS}, every $3$-dimensional real vector space of $4\times 4$ symmetric matrices contains a non-zero matrix with a multiple eigenvalue. In particular, some 
linear combination $A=c_0A_0+c_1A_1+c_2A_2$ has a double eigenvalue 
$\lambda$. Since $\lambda$ is real, $\lambda I-A$ has rank $2$.
This corresponds to a real node on the symmetroid
$X$. 
\end{proof}

The proof of \cite[Theorem~B]{FRS}  is based on general results by Adams 
concerning vector fields on spheres. 
In order to be self-contained, we include an elementary proof   for 
our special case.

\begin{proof}[Proof of double eigenvalues:]
 Suppose that for each non-zero $x=(x_0,x_1,x_2)$,
 the matrix $x_0A_0+x_1A_1+x_2A_2$ has four distinct real eigenvalues $\lambda_1(x)>\lambda_2(x)> \lambda_3(x)> \lambda_4(x)$.
  Let $V(x) = \bigl(v_1(x),v_2(x),v_3(x),v_4(x)\bigr)$ denote the 
  $4 \times 4$ rotation matrix whose rows are the corresponding
    eigenvectors of unit length. These are defined up to sign.
  Since $\RR^3\backslash \{0\}$ is simply connected, we can assume that the $v_i(x)$ depend continuously on $x$ and that $v_1(-x)=v_4(x)$, $v_2(-x)=v_3(x)$ for all $x\in \mathbb{S}^2$. We will show that this is impossible.
  
Consider the closed loop $\eta(t)=(\cos(2\pi t),\sin(2\pi t),0)$ in $\mathbb{S}^2$, 
 and let $\gamma(t)=V(\eta(t))$  be the corresponding loop of matrices in $SO(4)$.
 Since $\eta$ is contractible, also  $\gamma$ is contractible.
  We have $\,\gamma(t+\tfrac12)=J\cdot \gamma(t)\,$ for all $t\in [0,\tfrac12]$, where $J$ is the  permutation matrix that reverses
   the order of the rows of $\gamma(t)$.   We  derive a contradiction by showing that $\gamma$ is not contractible.

Let $\mathbb{S}^3\subset \HH=\{a+b\i+c\j+d\k\}$ denote the group of unit quaternions. 
The homomorphism $\pi:\mathbb{S}^3\times \mathbb{S}^3\to SO(4)=SO(\HH)$ given by 
$\pi(p,q)=(v\mapsto pvq^{-1})$ is a double covering, and it identifies $SO(4)$ with 
$\mathbb{S}^3\times \mathbb{S}^3/\!\pm(1, 1)$ (see e.g., \cite[p. 294]{Hat}). 
Furthermore, the identity 
$$\j(a+b\i+c\j+d\k)\i^{-1}=d+c\i+b\j+a\k$$
shows that the involution $J$ corresponds to left multiplication by the element $(\j,\i)\in 
\mathbb{S}^3\times \mathbb{S}^3$.
In particular, its square $(\j,\i)^2=(-1,-1)$ is in the kernel of $\pi$.

Since $\mathbb{S}^3\times \mathbb{S}^3$ is simply connected, the loop $\gamma$ lifts to a path $\tilde \gamma:[0,1]\to 
\mathbb{S}^3\times \mathbb{S}^3$. The endpoint $\tilde \gamma(1)$ must map to $\gamma(0)=\gamma(1)$, 
so $\tilde \gamma(1)=(\j,\i)^k \cdot \tilde \gamma(0)$ for some $k \in \{0,2\}$.
In fact, this $k$ is a homotopy invariant of $\gamma$ (by \cite[Prop. 1.30]{Hat}). Since $\gamma(t+\tfrac12)=J \cdot \gamma(t)$ for $t\in [0,\tfrac12]$ and $J\cdot \gamma(0)\neq \gamma(0)$, the lift of $\gamma$ satisfies 
$\tilde\gamma(\tfrac12+t)=(\j,\i)^m \cdot\tilde\gamma(t)$ for some $m\in \{1,3\}$.  Hence
 $$ \tilde\gamma(1)\,\,=\,\,(\j,\i)^m \cdot \tilde\gamma(\tfrac12)\,\,=\,\,(\j,\i)^{2m} \cdot \tilde\gamma(0)\,\,=\,\,(\j,\i)^{2}
 \cdot \tilde\gamma(0)\,\,=\,\,-\tilde \gamma(0).$$ 
 But, $\gamma$ being contractible implies
$\tilde \gamma(0)=\tilde \gamma(1)$. This contradiction completes the proof. 
\end{proof}

\begin{remark}\label{cubicspec} 
For \emph{cubic} symmetroids, we can argue in a similar way: If $x_0A_0+x_1A_1+x_2A_2$ has distinct real eigenvalues for every $x=(x_0,x_1,x_2)$, we may choose linearly independent orthonormal eigenvectors $v_1(x),v_2(x),v_3(x)$, so that $v_1(-x)=v_3(x)$ and $v_2(-x)=-v_2(x)$. Let $J\in SO(3)$ be the corresponding involution on the $v_i$. Fix the group $\mathbb{S}^3$ of unit quaternions.
There is a double covering $\pi:\mathbb S^3\to SO(3)$ which identifies $SO(3)$ with $\mathbb S^3/\pm1$. Explicitly, $\pi(q)=(v\mapsto qv\bar {q})$, where $v=a\i+b\j+c\k$. Under this identification,
 $J$ corresponds to left multiplication by $\frac1{\sqrt{2}}(\i+\k)\in \mathbb S^3$. If we define the loop $\gamma:[0,1]\to SO(3)$ as above, it lifts to a path $\tilde \gamma$ in $\mathbb S^3$ which satisfies  $\tilde \gamma(1)=(\frac1{\sqrt{2}}(\i+\k))^2 \tilde \gamma(0)=-\tilde \gamma(0)$. Again, this implies that $\gamma$ is not contractible and so the symmetroid has a real node.
\end{remark}

\begin{example}\label{nonspec} The hypothesis that the spectrahedron is non-empty is essential in
Lemma~\ref{lem:atleastone}. Consider the symmetroids, of degree $3$ and $4$
respectively, given by the symmetric matrices
\begin{small} $$ 
 \begin{bmatrix}x&
       y&
       z\\
       y&
       z&
       w\\
       z&
       w&
       {-x}\\
       \end{bmatrix}
\qquad        \hbox{and} \qquad
\begin{bmatrix}x&
       y&
       z&
       w\\
       y&
       z&
       {-x}&
       w\\
       z&
       {-x}&
       {-w}&
       z\\
       w&
       w&
       z&
       {-y}\\
       \end{bmatrix}.
$$
\end{small}
These two symmetroids have no real nodes. They lie in
 $\mathcal{S}_{\rm veryreal} \backslash \mathcal{S}_{\rm spec}$. $\hfill \diamondsuit$
\end{example}

Theorem \ref{thm:DI}
concerns transversal spectrahedral symmetroids.
These have ten rank-$2$ nodes as singularities.  Non-transversal spectrahedral symmetroids
have similar constraints.

\begin{proposition}\label{empty} A real singular point on a spectrahedral quartic symmetroid
 has rank $\leq 2$.
\end{proposition}

\begin{proof}  Fix a real singular point $p$ on the symmetroid $X$,
and let $L$ be a line through $p$ that contains a point $p^{\prime}$ in the interior of the spectrahedron.  The matrix $A(p^{\prime})$ is then definite  and therefore conjugate to the identity matrix. Hence, in the pencil $\lambda A(p^{\prime})-A(p)$, the singular matrices have corank equal to the multiplicity of the corresponding root of the equation 
\[\det (\lambda A(p^{\prime})-A(p))=0.\]
The multiplicity of a root equals the
multiplicity of the corresponding intersection point in $L \cap X$.
 Since $p \in {\rm Sing}(X)$,
 the root $\lambda =0$ has multiplicity at least $2$.
 This means ${\rm corank}(A(p)) \geq 2$, and hence 
 ${\rm rank}(A(p)) \leq 2$.
  \end{proof}

The spectrahedral hypotheses in Proposition \ref{empty} is necessary.
This can be seen from Theorem \ref{thm:gram} below:
if the univariate polynomial $p \in \R[t]$ is positive, then the spectrahedron ${\rm Gram}(p)$ 
is non-empty and its symmetroid has six complex nodes of rank $3$;
otherwise, ${\rm Gram}(p)$ is empty and the symmetroid
can have up to six real nodes of rank $3$.

\smallskip

Our next topic is the precise relationship between the conic and ramification cubics
seen in the  proof of Proposition \ref{prop:inside}.
 Let $f(x) = {\rm det}(A(x))$ and $e\in \R^4$ in the interior of 
 the spectrahedron $ S = S(f)$, so $A(e)$ is a definite matrix.
   Then,
for every $x\in \R^4$, all roots of the univariate polynomial $f(te+x)\in \R[t]$ are real. 
A cubic polynomial $h(x)$ \emph{interlaces} the quartic $f(x)$ with respect to $e$
if, for each $x$ in $\R^4$, the three roots $\beta_1,\beta_2,\beta_3$ of the polynomial $h(te+x)$ are
real and interlace the roots $\alpha_1\leq
\alpha_2 \leq \alpha_3  \leq \alpha_4$ of $f(te+x)$. By this we mean
\[ \alpha_1 \,\leq \,\beta_1 \,\leq \,\alpha_2 \,\leq \,\beta_2 \,\leq\, \alpha_3 \,\leq \, \beta_3\,\leq \,\alpha_4.\]
Theorem~3.3 of \cite{PV} states that the matrix $A(e)$ is definite if and only if the diagonal cofactors 
of the matrix $A(x)$ interlace $f$ with respect to the point $e$.  In our case, the ramification cubics 
are such diagonal cofactors.  This allows us to add to our extension of Cayley's Theorem.

\begin{proposition}\label{interlacing}
Let $p$ be a  $(1,1)$-node on a very real quartic symmetroid $X=V(f)$  with ramification $R_1 \cup R_2$
and no line on $X$ through $p$. Let $C_p$ be the real conic totally tangent to $R_1\cup R_2$.
 The spectrahedron  $S $ is non-empty if
 and only if the conic $C_p$ interlaces the two ramification cubics $R_1$ and $R_2$ with respect to a common point.
 In this case the image of $S$ under projection from $p$ is the intersection of the convex hull of the ovals of $R_1$ and $R_2$. 
 \end{proposition}

\begin{proof}
($\Rightarrow$) We write $f$ 
in the form \eqref{eq:symmatrix1}  over $\R$.
The ramification cubics $R_1$ and $R_2$
 are the cofactors  $F_{11}$ and $F_{22}$. 
Let $e\in \R\PP^2$ be an interior point in the image of the spectrahedron $S$ 
 under projection from $p$. 
 For some $e_0\in \R$  the point $(e_0,e)$ lies in the interior of $S$ in $\R \PP^3$.
Then the $3\times 3$ submatrix defined by the rows and columns $\{2,3,4\}$ is also 
positive definite at $e$ and the quadric $q$ defining $C_p$ is a diagonal cofactor of this matrix.  
By \cite[Theorem 3.3]{PV}, the conic $q$ interlaces the 
$3\times 3$ determinant $F_{11}$ with respect to $e$.
Taking instead the $3\times 3$ principal submatrix indexed by $\{1,3,4\}$, 
we see that $C_p$ interlaces $R_2$. 

($\Leftarrow$)  Let $p = (1:0:0:0)$ and
suppose that $C_p$ interlaces both $R_1$ and $R_2$ with respect to $e \in \R^3$. Certainly,
$C_p$, $R_1$, and $R_2$ then have real points. Theorem~\ref{thm:cayley} says that $V(f)$ is very real and
$p$ is a $(1,1)$-node. Its proof  gave a real determinantal representation $A(x)$, up to sign, for $f$
 via \eqref{eq:DetRep}.
It also produced a symmetric $4 {\times} 4$ matrix  $M(x)$ with linear entries in
$\R[x_1,x_2,x_3]$ such that $R_1 = V((M^{\rm adj})_{11} )$, $R_2 = V((M^{\rm adj})_{22}) $,
and $C_p = V( \det(M_{34}) )$.
By the interlacing hypothesis and \cite[Theorem 3.3]{PV}, 
the two principal $3 {\times} 3$ submatrices of $M$ indexed by $\{2,3,4\}$ and $\{1,3,4\}$ 
are both definite  at $e$. 
Since these two definite submatrices of $M(e)$ share a $2\times 2$ diagonal block,
they are either both positive definite or both negative definite. Without loss of 
generality we take them to be positive definite.  In particular, both of the 
$3\times 3$ determinants $(M^{\rm adj})_{11}$ and $(M^{\rm adj})_{22}$ are positive
at the point $e$.

Consider the matrix $A(x)$ given in \eqref{eq:DetRep}, where $m_{jk}$ is the $(j,k)$-th entry of $M(x)$. 
We will show that, for some $e_0\in \R$, the matrix $A(e_0,e)$ is positive definite. 
Using the notation in the proof of Theorem~\ref{thm:cayley}, 
the quadratic polynomial $\det(A(x_0,e)) \in \R[x_0]$ equals 
\begin{equation}
\label{eq:quadratic}
\det(A(x_0,e)) \;\;=\;\;  c^3\cdot \left(- q(e) x_0^2 + 2g(e)x_0 + \Delta(e) \right) .
\end{equation}
Its discriminant is positive:
\[4 g(e)^2 + 4 q(e)\Delta(e)  \;\; =\;\; 4\cdot (M^{\rm adj}(e))_{11}\cdot (M^{\rm adj}(e))_{22} \;\; > \;\; 0.\]
Thus, the quadratic (\ref{eq:quadratic}) changes signs,
and it is positive at some  point $x_0 = e_0\in \R$.
The matrix $A(e_0,e)$ is positive definite. Indeed, the lower
$3\times 3$ submatrix is the same as that of $M(e)$. Thus, the principal minors indexed by
 $\{4\}$, $\{3,4\}$ and $\{2,3,4\}$ are all positive at $(e_0,e)$. 
Since the determinant is positive as well, the matrix $A(e_0,e)$ is positive definite. \smallskip

The above argument shows that for $e\in \R^3$, there exists $e_0\in \R$ so that the matrix $A(e_0,e)$
is definite if and only if the principal $3\times 3$ submatrices defining $R_1$ and $R_2$ are both definite at 
 $e$.  Thus the projection of the spectrahedron $S$ is precisely the intersection of the two cubic 
 spectrahedra bounded by the ovals of $R_1$ and $R_2$. This gives the last statement.
\end{proof}

In our discussion so far we have disallowed
lines through the projection node on $X$.
We close Section \ref{sec:view} with a discussion of
  symmetroids with lines, and non-transversal symmetroids.

\begin{lemma}\label{line} Let $L$ be a line in a nodal quartic symmetroid $X$. 
 Then $X$ has three nodes on $L$, and all three nodes have rank $2$.
Furthermore, if $X$ is very real, then either all three are $(1,1)$-nodes or one is a $(1,1)$-node and the two other are $(2,0)$-nodes. 
\end{lemma}

\begin{proof} 
The matrices in the pencil parameterized by $L$ all have rank $\leq 3$, so they must share a common kernel vector.  The discriminant of the pencil is a cubic form that vanishes where the matrices have rank at most $2$.  Geometrically, this cubic form defines the common intersection of the cubic curves residual to the line $L$ in plane sections of $X$,  i.e. the cubic curves that together with $L$ form plane sections. 
Furthermore, for a node the singularity of the general plane section through the line is also nodal.
   Thus, if $X$ is nodal, the general residual cubic curve must intersect the line transversally at three points, and each point is a rank $2$ node on $X$.   
If $X$ is very real, one of the eigenvalues of the matrices parameterized by $L$ changes sign at each node, so if there is a semidefinite rank $3$ point on the line, there are  exactly two semidefinite rank $2$ nodes, otherwise all three are $(1,1)$-nodes.   
\end{proof}

 \begin{corollary}
 Let $X$ be a transversal spectrahedral quartic symmetroid that contains a real line $L$.  Let $p$ be a real node on $L$.  Then the two cubic ramification curves $R_1$ and
 $R_2$ under the projection from $p$ intersect on the conic $C_p$, the image of the tangent cone.  
 \end{corollary}
 
\begin{proof}
 A line through a real node $p$ in a nodal symmetroid lies in the tangent cone at the node.  
Here the conic $C_p$ contains the image
  of the other two rank $2$-nodes on this line.
 \end{proof}

\section{ Spectrahedral treats}\label{sec:treats}

This section features three special families of spectrahedra.
Their names were chosen to highlight the connection between
convex optimization \cite{BPT} and classical algebraic geometry \cite{Do}.
{\em Toeplitz spectrahedra} have a curve of singular points in their boundary.  
{\em Sylvester spectrahedra}
are transversal but contain $10$ lines in their symmetroids. 
{\em  Kummer spectrahedra}
 are nodal but not transversal. Their characterization in Theorem
\ref{thm:gram} may be  of independent interest.

\begin{figure}[h]
\vskip -0.6cm
\begin{center}
\includegraphics[scale=0.25]{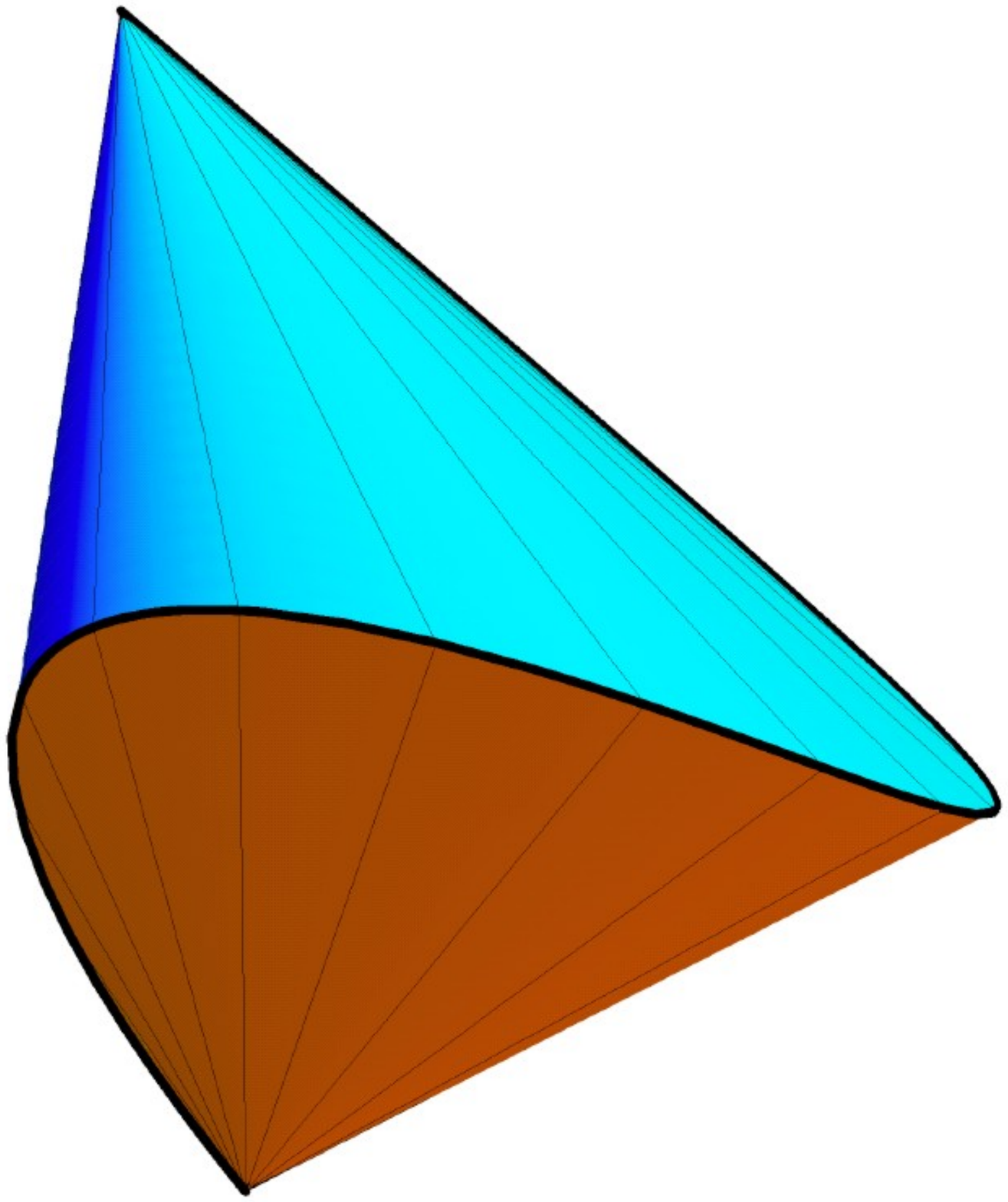}
\includegraphics[scale=0.23]{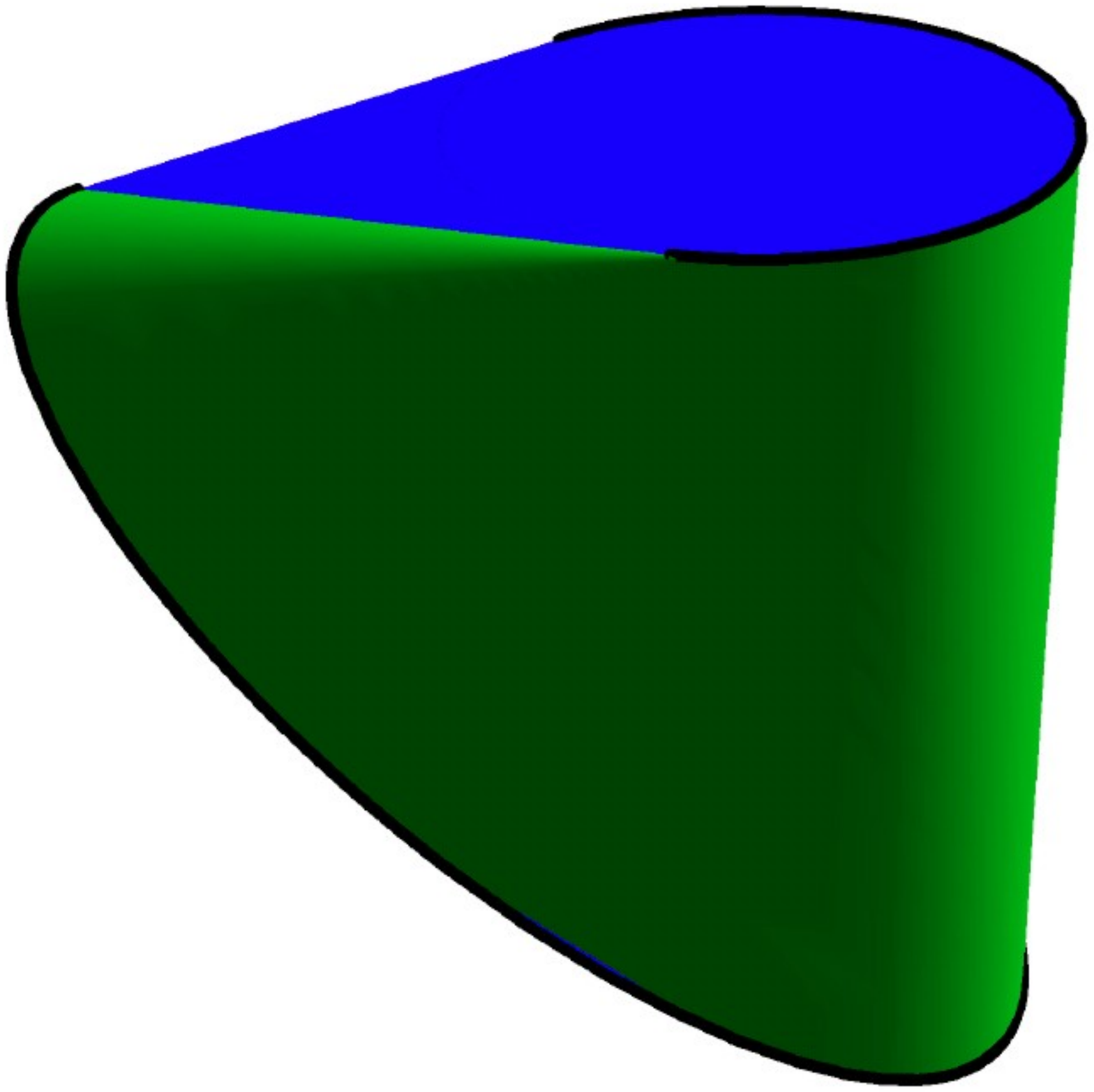}
\end{center}
\vskip -1.4cm
\caption{\label{fig:toeplitz}
Toeplitz spectrahedron and its dual convex body.
}
\end{figure}

\begin{example} \label{ex:tope}
The {\em Toeplitz spectrahedron} is the following convex set
 in affine $3$-space:
\begin{equation}
\label{eq:toeplitz}
\biggl\{ (x,y,z) \in \R^3 \,: \,
\begin{small}
\begin{bmatrix}
  1 & x & y & z \\
     x & 1 & x & y \\
     y & x & 1 & x \\
     z & y & x & 1 
\end{bmatrix}
\end{small}
 \,\succeq \, 0 \,\biggr\}.
\end{equation}
The determinant of the given $4 \times 4$-matrix factors as
$$ (x^2+2xy+y^2-xz-x-z-1)(x^2-2xy+y^2-xz+x+z-1). $$
The  Toeplitz spectrahedron  (\ref{eq:toeplitz})
is the convex hull of the {\em cosine moment curve}
$$ \bigl\{ \bigl( {\rm cos}(\theta), {\rm cos}(2 \theta), {\rm cos}(3 \theta)\bigr) \,:\,
\theta \in [0, \pi] \,\bigr\} . $$
The curve and its convex hull are shown on the left in Figure \ref{fig:toeplitz}.
The two endpoints, $(x,y,z) = (1,1,1) $ and
$(x,y,z) = (-1,1,-1)$,
correspond to rank $1$ matrices. All other points on the curve
have rank $2$.
To construct the Toeplitz spectrahedron geometrically, we form the cone
from each endpoint over the cosine curve, and we intersect
these two quadratic cones. 
The two cones intersect along this curve and the line through the endpoints of the cosine curve.

Shown on the right in Figure \ref{fig:toeplitz} is the dual convex body.
It is the set of trigonometric polynomials 
$1+a_1\cos(\theta) + a_2\cos(2\theta) +a_3\cos(3\theta) $ that are nonnegative on  $[0,\pi]$.
 This convex body is not a spectrahedron because it has a non-exposed edge
(cf.~\cite[Exercise 6.13]{BPT}). 
\hfill $\diamondsuit$ 
\end{example}

In the earlier sections we focused on quartic spectrahedra
that are transversal, given by generic symmetric
matrices $A_0,A_1,A_2,A_3$, so they have precisely
$10$ nodes.  Theorem \ref{thm:DI} reveals
their distribution on and off the spectrahedron.
Toeplitz spectrahedra are far from being transversal.
Indeed, they are singular along an entire curve of rank-2 matrices; the intersection of the two quadratic cones above. Moreover, the spectrahedron is the convex hull of that curve.

%They are obtained as the intersection of
%two quadratic cones. Typically,
%symmetroids that are defined in this manner
%contain no rank-$1$ matrices, but
%the set of rank-$2$ matrices is an entire curve.
%This curve is the singular locus of the symmetroid,
%and it is generally a quartic curve of genus one.
%The spectrahedron
%is the convex hull of that curve.

\begin{figure}[h]
\vskip -0.5cm
\begin{center}
\includegraphics[scale=0.21]{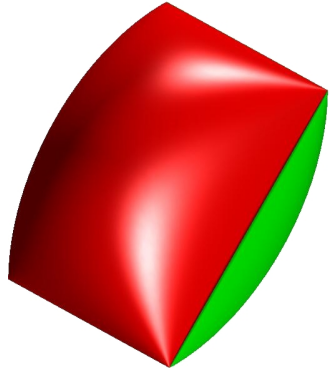} \qquad
\end{center}
\vskip -1.2cm
\caption{\label{fig:pillow}
The pillow is bounded by an irreducible quartic surface.}
\end{figure}

\begin{example}
\label{ex:pillow}
Figure \ref{fig:pillow} shows
a spectrahedron that appears to be reducible, but it is not:
\begin{equation}
\label{eq:nontoeplitz}
\biggl\{ (x,y,z) \in \R^3 \,: \,
\begin{small}
\begin{bmatrix}                                                             
  1 & x & 0 & x \\                                                          
     x & 1 & y & 0 \\                                                       
     0 & y & 1 & z \\                                                       
     x & 0 & z & 1                                                          
\end{bmatrix}
\end{small}     \,\, \succeq 0 \,\, \biggr\}. 
\end{equation}
This spectrahedron was featured in \cite[\S 5.1.1]{BPT}
where it was called {\em the pillow}. This symmetroid
is an irreducible quartic surface that is singular along two lines in the plane at infinity and in the 
four corners of the pillow.  These four corners are coplanar and form the intersection of two pairs of parallel lines that lie in the symmetroid.
\hfill $\diamondsuit$
\end{example}

\smallskip

A {\em Sylvester spectrahedron} is the derivative, 
in the sense of Renegar \cite{Ren} and Sanyal \cite{San}, of a 
$3$-dimensional polytope $P$
defined by five linear inequalities.
Using homogeneous coordinates, the inequalities are
given by linear forms  $\ell_1,\ell_2, \ell_3, \ell_4, \ell_5 \in \R[x_0,x_1,x_2,x_3]$.
The  symmetroid 
is the classical polar, with respect to any interior point of $P$,
of the quintic  
$\{ x \in \C \PP^3: \ell_1(x) \ell_2(x) \ell_3(x) \ell_4(x) \ell_5(x) = 0 \}$.
Following \cite[Theorem 1.1]{San},
we use the matrix
\begin{equation}
\label{eq:sylvester}
A(x) \quad = \quad 
\begin{bmatrix}
\ell_1(x)+\ell_5(x) & \ell_2(x) & \ell_3(x) & \ell_4(x) \\
  \ell_1(x) &  \ell_2(x)+\ell_5(x) & \ell_3(x) & \ell_4(x) \\
   \ell_1(x) &  \ell_2(x) & \ell_3(x)+\ell_5(x) & \ell_4(x) \\
  \ell_1(x) &  \ell_2(x) & \ell_3(x) & \ell_4(x)+\ell_5(x) 
  \end{bmatrix}
\end{equation}
to define the Sylvester spectrahedron.
If no four of the $\ell_i(x)$ are linearly dependent, and
none of the five  inequalities $\ell_i(x) \geq 0 $ is redundant,
then the polytope $P$ is a triangular prism.
Figure \ref{fig:toblerone}
shows a Sylvester spectrahedron that is derived from
a triangular prism $P$.

\begin{figure}[h]
\vskip -0.4cm
\begin{center}
\parbox[m]{4in}{\includegraphics[width=3.5in]{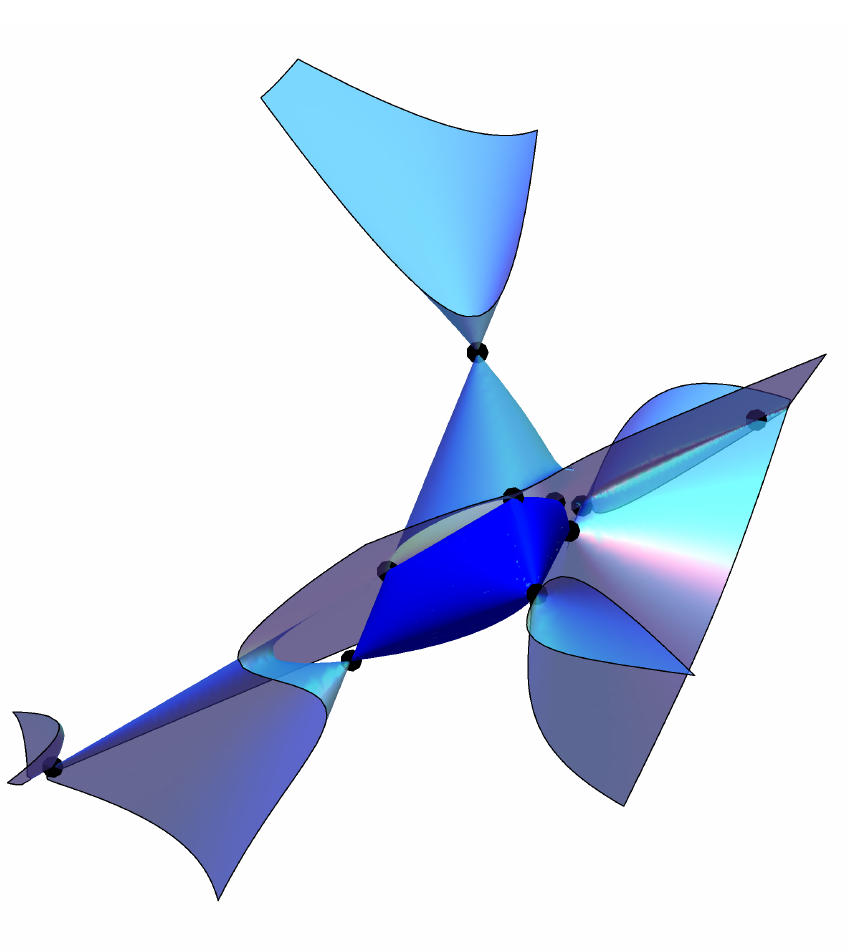} }
\parbox[m]{2in}{\includegraphics[width=2in]{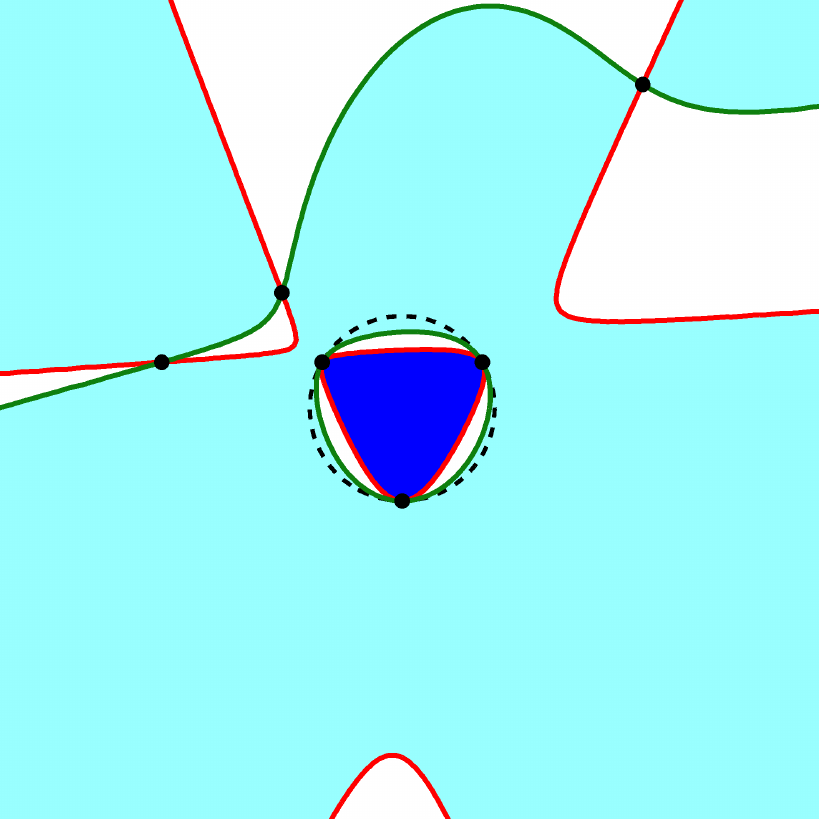} }
\end{center}
\vskip -1.2cm
\caption{\label{fig:toblerone}
Sylvester spectrahedron derived from a triangular prism and a nodal projection.
}
\end{figure}

The name ``Sylvester spectrahedron'' is a reference to
{\em Sylvester's Pentahedral Theorem}. This theorem states that
the polynomial defining a general cubic surface in $\C \PP^3$
admits a unique representation,
over the complex numbers, as a sum of five powers of linear forms:
$$ f(x) \,\, = \,\, \ell_1(x)^3 + \ell_2(x)^3 + \ell_3(x)^3 + \ell_4(x)^3 + \ell_5(x)^3 . $$
The Hessian $\,\bigl( \partial^2 f /\partial x_i \partial x_j \bigr)\,$
of this polynomial is the symmetric $4 \times 4$-matrix
 $\,\sum_{i=1}^5 \ell_i H_i  $, where $H_i $ is the rank-$1$ matrix
$(\nabla \ell_i)^T (\nabla \ell_i)$. In suitable coordinates, this
Hessian matrix is precisely the symmetric matrix $A(x)$ given in (\ref{eq:sylvester}).
 Its determinant  is the quartic form
$$ {\rm det}(A(x)) \quad = \quad \left(
\frac{1}{\ell_1(x)} + 
\frac{1}{\ell_2(x)} +
\frac{1}{\ell_3(x)} +
\frac{1}{\ell_4(x)} +
\frac{1}{\ell_5(x)} \right) \ell_1(x) \ell_2(x) \ell_3(x) \ell_4(x) \ell_5(x). $$

The Sylvester spectrahedron is transversal, i.e.~its singularities exhibit the generic behavior.
The corresponding Sylvester symmetroid in $\C \PP^3$
has precisely ten nodes,
namely the intersection points $\{\ell_i(x) = \ell_j(x) = \ell_k(x) = 0\}$.
All ten nodes are real, so $\rho = 10$.
However, the symmetroid also contains the ten lines
$\{\ell_i(x) = \ell_j(x) = 0\}$. Each line contains three of the nodes.
Compare this to the  derivative of a tetrahedron shown on the left in
 Figure \ref{fig:redyellow}.
 
  Starting from a triangular prism $P$,
the Sylvester spectrahedron has the six vertices and nine edges in its boundary.
These account for six of the nodes and nine of the lines.  The
remaining four nodes of the symmetroid are outside the spectrahedron.
Thus  $(\rho ,\sigma) = (10,6)$ 
in the notation of the census of transversal types in Section 2.
The spectrahedron is shown in dark blue in  Figure \ref{fig:toblerone}. It
 wraps around the polytope $P$ and  contains its edge graph.

The {\em Sylvester symmetroid}  is the quartic surface in 
$\C \PP^3$ defined the vanishing of the determinant (\ref{eq:sylvester}).
The projection of the surface from one of the nodes maps six nodes in pairs to three nodes of a triangle, and the remaining three nodes to the three intersection points of this triangle with a line.    These six points in the plane are the only intersection points of the pair of cubics, the ramification curve of the projection.  The two cubics are tangent at the nodes of the triangle, and there is a $4$-dimensional family of such pairs of cubics with a totally tangent conic, corresponding to the family of cubic forms with parameters $\ell_1,...,\ell_5$.

\begin{figure}[h]
%\vskip -2cm
\begin{center}
\includegraphics[scale=0.48]{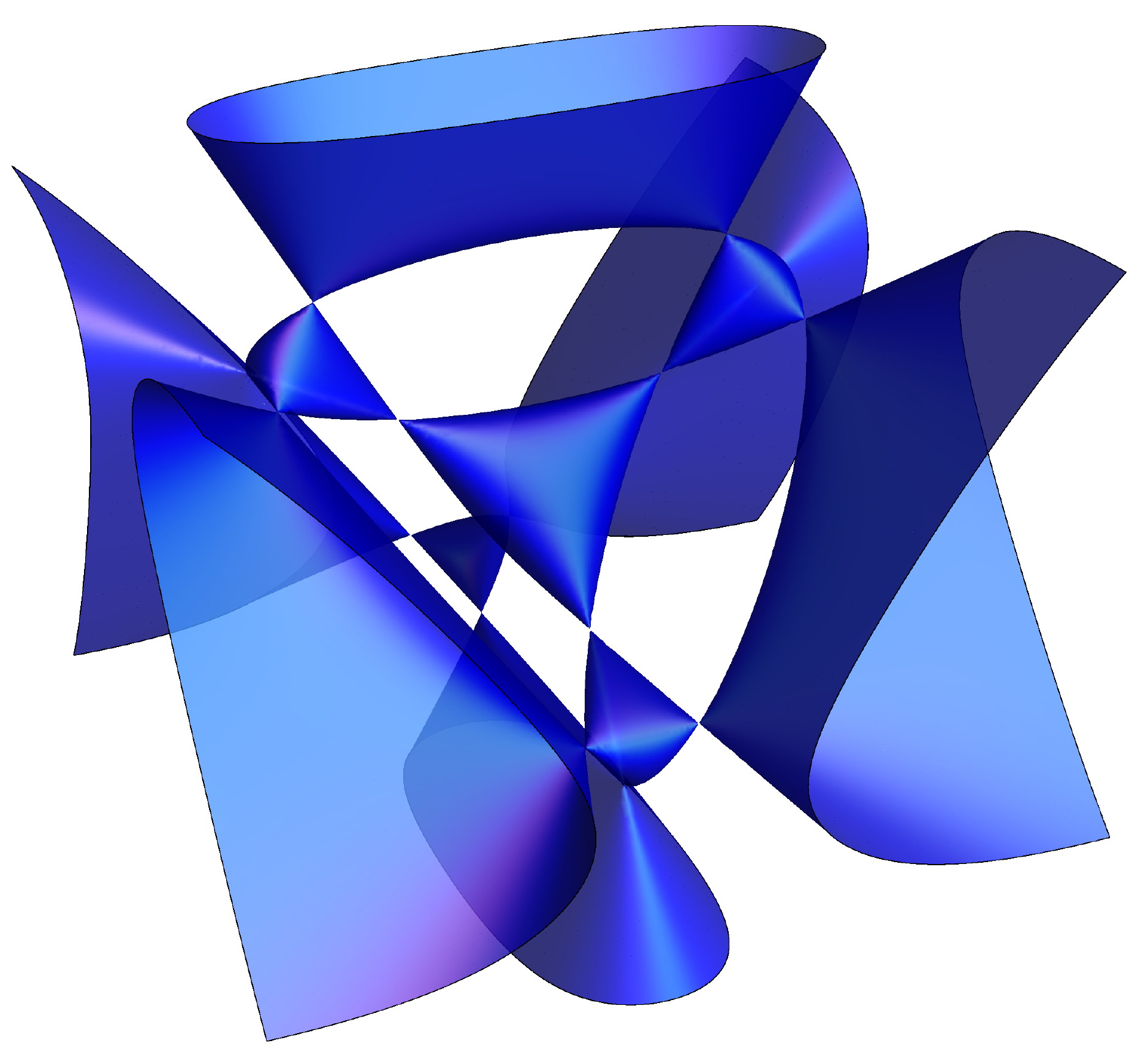}
\end{center}
%\vskip -2.1cm
\caption{\label{fig:kummer}
The Kummer surface is a symmetroid with $16$ nodes.
}
\end{figure}

\bigskip

A perennial favorite among quartics in $\C \PP^3$ is the {\em Kummer surface}.
On the one hand they are the quartic surfaces with a maximal number of nodes, 
namely sixteen \cite{hudson, J}. 
On the other hand, any Kummer surface is the quotient of an abelian surface by its involution, 
so it has $16$ nodes, one for each $2$-torsion point. All $16$ nodes can be real,
as seen in Figure \ref{fig:kummer}.
The compact body  seen in the center may or may not be convex.
Even if it is convex, it cannot be a spectrahedron, as we shall see
in Corollary~\ref{cor:specempty}.

 The following  representation of Kummer surfaces
as symmetroids appears in
 equation (7) on page 143 in Section 41 of Coble's book \cite{Co},
 where this is attributed to H.F.~Baker:
\begin{equation}\label{eq:KummerMatrix}
A(x) \, = \,  \begin{bmatrix}
     a_6 x_3 &      -3 a_5 x_3  &  3 a_4 x_3+x_2 &  \! -a_3 x_3-x_1 \\
     -3 x_3 a_5 &  9 a_4 x_3-2 x_2 &  \! -9 a_3 x_3+x_1 &  3 a_2 x_3+x_0 \\
  \! 3 a_4 x_3+x_2 & \!  -9 a_3 x_3+x_1 &  9 a_2 x_3-2 x_0 &    -3 a_1 x_3 \\
    -a_3 x_3-x_1 &    3 a_2 x_3+x_0 &      -3 a_1 x_3 &       a_0 x_3
   \end{bmatrix}.
\end{equation}
Consider the following sextic polynomial in one variable $t$:
\begin{equation}
\label{eq:unisex}
p(t) \quad = \quad a_0 - 6 a_1 t + 15 a_2 t^2 - 20 a_3 t^3 + 15 a_4 t^4 - 6 a_5 t^5 + a_6 t^6 .
\end{equation}
This polynomial is related to the determinantal representation $A(x)$ as follows:
\begin{equation}
\label{eq:QuadricInCubics}
x_3 p(t) \,\,\, = \,\,\,\,
\begin{bmatrix} t^3 \! & \! t^2 \! & \! t \! & \! 1 \end{bmatrix}
\cdot A_3 x_3 \cdot
\begin{bmatrix} t^3 \! & \! t^2 \! & \! t & \! 1 \end{bmatrix}^T
\,\,\, = \,\,\,
\begin{bmatrix} t^3 \! & \! t^2 \! & \! t \! & \! 1 \end{bmatrix}
\cdot A(x) \cdot
\begin{bmatrix} t^3 \! & \! t^2 \! & \! t & \! 1 \end{bmatrix}^T.
\end{equation}
Suppose that $p(t)$ has six distinct complex roots $u_1,u_2,\ldots,u_6$.
Then  the symmetroid $\{{\rm det}(A(x)) = 0\}$  has $16$ isolated nodes.
Ten of these nodes correspond to matrices of rank~$2$, and we shall
describe these below.  The other six nodes are
$(u_i^2:u_i:1:0)$ and these correspond to rank~$3$ matrices.
All six lie in the plane $\{x_3 = 0\}$, so this is
one of the $16$ planes in the $16_6$ configuration associated
with the Kummer surface. The abelian surface corresponding to the
Kummer surface  $\{{\rm det}(A) = 0\}$ is the Jacobian of the
genus $2$ curve $y^2 = p(t)$, obtained as the double cover
of $\C \PP^1$ ramified at $u_1,u_2,u_3,u_4,u_5,u_6$.   

Let $\mathcal{K}$ denote the closure 
in the space $\C \PP^{34}$ of all quartics,
of the locus of all quartic Kummer surfaces in $\C \PP^3$.
 The Kummer locus $\mathcal{K}$ is an irreducible variety of dimension $18$.
\begin{proposition} \label{prop:KinS}
The variety $\mathcal{K}$ of Kummer surfaces lies in
the variety $\mathcal{S}$ of symmetroids.
\end{proposition}

\begin{proof}
Every smooth curve of genus $2$
admits a representation as the double cover
of $\C \PP^1$ defined by $y^2 = p(t)$, where $p(t)$ has six distinct roots. This gives rise to a symmetroid surface $A(x)$ as above. Now it is a classical fact that  Jacobians of such curves taken modulo involution are dense in the space $\mathcal{K}$
of Kummer surfaces, so it follows that $\mathcal{K} \subset \mathcal{S}$.
\end{proof}

\begin{remark} Kummer symmetroids have the following interpretation in terms of quadrics. 
Identifying symmetric matrices $A$ with quadrics $z^TAz$, the family $A(x)$ corresponds to a $4$-dimensional vector space of quadrics $W$. The symmetroid is the subset $X\subset  \PP(W)$ parameterizing the singular members of the family. 
The Kummer symmetroids correspond to those vector spaces $W$ containing a 3-dimensional subspace $V$ whose quadrics define a twisted cubic curve. 
Indeed, if $X$ is a Kummer surface, then in the notation of \eqref{eq:KummerMatrix}, the quadrics of  $\langle A_0,A_1,A_2\rangle$ define a twisted cubic curve; $V=\langle z_2^2-z_1z_3, z_0z_3-z_1z_2, z_1^2-z_0z_2\rangle$. Conversely, given $W$ with a vector subspace $V\subset W$ defining a twisted cubic, the plane $\PP(V)\subset \PP(W)$ intersects the symmetroid in a double conic $\{q^2=0\}$ which parameterizes rank three quadrics that contain the twisted cubic curve.
Now if $\PP(V)=\{x_3=0\}$ and the symmetroid is defined by $x_3^4+ax_3^3+bx_3^2+cx_3+q^2$, with $a,b,c,q\in \C[x_0,x_1,x_2]$, then $\{x_3=c=q=0\}$ defines the six rank $3$-nodes on the symmetroid. 
If the symmetroid $X$ is nodal, it has ten additional rank 2-nodes and is Kummer.  \end{remark}

A {\em Kummer spectrahedron} is an element of the semialgebraic set
$\,\mathcal{K}_{\rm spec} = \mathcal{K} \cap \mathcal{S}_{\rm spec}$.
Such a spectrahedron has the following interpretation in convex algebraic geometry.
Given a real vector space $V$ we write $S^d V$ for 
the $d$-th symmetric power of its dual.  Consider the map
\begin{equation}\label{eq:psi}
 \psi \,:\, S^2 S^3 \R^2 \,\rightarrow \,S^6 \R^2, \,\, \ \ \  M \,\mapsto \,
\begin{bmatrix} t^3 \! & \! t^2 \! & \! t \! & \! 1 \end{bmatrix}
\cdot M \cdot
\begin{bmatrix} t^3 \! & \! t^2 \! & \! t & \! 1 \end{bmatrix}^T, 
\end{equation}
which takes quadrics in binary cubics to the corresponding binary sextics.
Elements in the $10$-dimensional space
 $S^2 S^3 \R^2$ are identified with symmetric
 $4 \times 4$-matrices $M$, and elements in the
 $7$-dimensional space $S^6 \R^2$ are identified
 with univariate polynomials $p(t)$ of degree $6$.

The kernel of the map $\psi$ is $3$-dimensional.
Hence the fiber of $\psi$ over any polynomial
$p(t)$ is a $3$-dimensional affine space.
That affine space contains the
{\em Gram spectrahedron} 
$$ {\rm Gram}(p) \,\, = \,\, \bigl\{ M \in 
S^2 S^3 \R^2  \,\,:\,\,
M \,\text{positive semidefinite \ and} \,\,
\psi(M) = p \bigr\}. $$
The term ``Gram spectrahedron'' was coined
in \cite[\S 6]{PSV}, where this  was
studied for ternary quartics $p$. The points
in ${\rm Gram}(p)$  correspond to 
sum of squares representations of $p$ over $\R$.
Indeed, a rank $r$ matrix $M$ in ${\rm Gram}(p)$
has a real Cholesky decomposition
$M = N \cdot N^T$ where $N$ has format $4 \times r$, and this
translates into a sum of squares representation
\begin{equation}
\label{eq:SOSp}
 p(t) \,\,= \,\, 
(\begin{bmatrix} t^3 \! & \! t^2 \! & \! t \! & \! 1 \end{bmatrix} N)
\cdot 
(\begin{bmatrix} t^3 \! & \! t^2 \! & \! t \! & \! 1 \end{bmatrix} N)^T .
\end{equation}
Clearly, in one variable, such a representation exists if and only if
$p(t)$ is non-negative on $\R$.

\begin{lemma}
The Gram spectrahedron ${\rm Gram}(p)$
of the univariate sextic $p(t)$ in (\ref{eq:unisex})
is affinely isomorphic  to the Kummer spectrahedron 
defined by the $4 \times 4$-matrix $A(x)$  in (\ref{eq:KummerMatrix}).
\end{lemma}

\begin{proof}
The kernel of the map $\psi$ in \eqref{eq:psi} is spanned by the matrices $A_0,A_1,A_2$
appearing in the representation  \eqref{eq:KummerMatrix}.
We identify ${\rm ker}(\psi)$ with the
affine space $\R^3$ obtained from $\R \PP^3$ by  setting $x_3 = 1$.
With this, the spectrahedron defined by our matrix $A(x)$  consists of
all positive semidefinite matrices in the fiber of $\psi$ over
$p(t) = \psi(A_3)$.
\end{proof}

Suppose that the sextic $\,p(t) = \prod_{i=1}^6 (t-u_i)\,$ has six distinct roots in $\C$.
Then there are ten distinct rank $2$ representations (\ref{eq:SOSp}).
These correspond to the ten distinct representations of $p(t)$
as a sum of two squares  over $\C$.
Explicitly, we have
 $\,p =  (q/2)^2 - (r/2)^2 $, where
$$ \begin{matrix}
q(t) & = & 2  t^3-( u_1{+} u_2 {+} u_3 {+} u_4 {+} u_5 {+} u_6)  t^2+( u_1   u_2+ u_1   u_3+ u_2   u_3+ u_4   u_5+ u_4   u_6+ u_5   u_6)  t \\ & & - u_1   u_2   u_3- u_4   u_5   u_6 , \\
r(t)  & = &   ( u_1+ u_2+ u_3- u_4- u_5- u_6)  t^2-( u_1   u_2+ u_1   u_3+ u_2   u_3- u_4   u_5- u_4   u_6- u_5   u_6)  t \\ & & + u_1   u_2   u_3- u_4   u_5   u_6.
\end{matrix}
$$
Note that the number of distinct $\{q,r\}$ is ten,
 one for each
partition $\{\{u_1,u_2,u_3\}, \{u_4,u_5,u_6\} \}$
of the six roots  into two triples.
If all roots are real then all ten formulas express
$p$ as a difference of two squares over $\R$.
As seen in Figure \ref{fig:kummer}, the
corresponding Kummer symmetroid has $16$ real nodes,
 six of rank $3$ and ten of rank $2$.
All rank $2$ nodes have signature $(1,1)$.

 Finally, let us consider the case when $p(t)$ is strictly positive,
 so over $\R$ we can write
 $$ p(t) \,\, = \,\, 
((t-a)^2 + b^2) \cdot
((t-c)^2 + d^2) \cdot
((t-e)^2 + f^2) . $$
The Kummer spectrahedron ${\rm Gram}(p)$ is 
non-empty and three-dimensional. We claim
that it has precisely four nodes in its boundary.
They correspond to 
the following representations:
$$
 \begin{matrix}
p(t) &  = &  \bigl( (t -a) (t -c) (t -e) - (t -a) d f -  (t -c)b f - b d (t -e) \bigr)^2 \quad  \\   & & \quad + \,\,
 \bigl( b d f   - b (t -c) (t -e) - (t -a) d (t -e) - (t -a) (t -c) f \bigr)^2 \\
& = & \bigl( (t-a) (t-c) (t-e) - (t-a) d f + (t-c) bf + b d (t-e) \bigr)^2 \quad \\ & & \quad + \,\,
 \bigl(  b d f   - b (t-c) (t-e) + (t-a) d (t-e) + (t-a) (t-c) f \bigr)^2 \\
& = & \bigl( (t-a) (t-c) (t-e) + (t-a) d f -  (t-c)b f + b d (t-e) \bigr)^2 \quad \\ & & \quad + \,\,
 \bigl(  b d f   + b (t-c) (t-e) - (t-a) d (t-e) + (t-a) (t-c) f \bigr)^2 \\
& = & \bigl( (t-a) (t-c) (t-e) + (t-a) d f + (t-c)b f - b d (t-e) \bigr)^2 \quad \\ & & \quad + \,\,
 \bigl(  b d f   + b (t-c) (t-e) + (t-a) d (t-e) - (t-a) (t-c) f \bigr)^2 .
 \end{matrix}
$$
These formulas arise from the four partitions 
$\{\{u_1,u_2,u_3\}, \{u_4,u_5,u_6\} \}$
of the set
$$ \{u_1,u_2,u_3, u_4,u_5,u_6\} \,\,= \,\,
\{ a + i b, a-ib, c+id, c-id, e+if, e-if\} $$
which satisfy $
(u_1-a)^2 + b^2 =
(u_2-c)^2 + d^2 =
(u_3-e)^2 + f^2 = 0$.
The other six partitions into two triples also give formulas
$p=  (q/2)^2 - (r/2)^2 \,$ but these
are not  defined over $\R$.
For instance, if
$u_1 = a+bi , u_2 = e+fi, u_3=e-fi $ and $
u_4 = a-bi , u_5 = c+di, u_6 = c-di $
then
$$
\begin{matrix}
q(t) &=& 2 t^3+(-2 a-2 c-2 e) t^2+(c^2+d^2+e^2+f^2+2 a c+2 a e) t-a e^2-a d^2-a c^2-a f^2 \\
& &    + \,(2 b e-2 b c) t+b c^2-b f^2+b d^2-b e^2) \cdot i ,\\
r(t) &=& (2 e-2 c) t^2+(-e^2+d^2-f^2-2 a e+c^2+2 a c) t+a f^2+a e^2-a c^2-a d^2 \\
& &     + \, (2 b t^2+(-2 b e-2 b c) t+b c^2+b f^2+b d^2+b e^2) \cdot i.
    \end{matrix}
$$    

Of course, all of the above decompositions can be translated into symmetric $4 \times 4$-matrices $A(x)$ of rank $2$.
We summarize our discussion on Gram spectrahedra of univariate sextics:

\begin{theorem} \label{thm:gram}
The spectrahedron ${\rm Gram}(p)$ is non-empty
if and only if $p(t)$ is non-negative. Suppose this holds
and $p(t)$ has simple roots in $\C$.
Then ${\rm Gram}(p)$ is three-dimensional, and its boundary contains
precisely four nodes, corresponding
  to the four representations of $p(t)$
as sum of two squares over $\R$. The algebraic boundary of
$ {\rm Gram}(p)$ is a Kummer surface,
associated with the genus $2$ curve $\{y^2 = p(t)\}$.
 It contains 
$12$ further complex nodes, six
of rank $2$ in the affine space $\{x_3 = 1\}$
and six of rank $3$ in the hyperplane at
infinity $\{x_3 = 0\}$.
\end{theorem}

Proposition~\ref{empty}
confirms  that nothing seen in
 Figure \ref{fig:kummer} can be a spectrahedron:

\begin{corollary} \label{cor:specempty}
If a Kummer surface has $\geq 5$ real nodes
then its spectrahedron is empty.
\end{corollary}

\bigskip
\bigskip

\noindent
{\bf Acknowledgements.}\\
%John Christian Ottem .....
%Kristian Ranestad .....
 Bernd Sturmfels was supported by the NSF (DMS-0968882)
     and the Max-Planck Institute f\"ur Mathematik in Bonn.
Cynthia Vinzant was supported by an NSF postdoc (DMS-1204447).

\bigskip
\bigskip

\footnotesize
\noindent {\bf Authors' addresses:}

\noindent John Christian Ottem, 
University of Cambridge, CB3 0WA,
UK, {\tt J.C.Ottem@dpmms.cam.ac.uk}

\noindent Kristian Ranestad, University of Oslo, Postboks 1053 Blindern, 0316 Oslo, Norway,
{\tt ranestad@math.uio.no}

\noindent Bernd Sturmfels,  University of California, Berkeley, CA 94720-3840,
USA, {\tt bernd@berkeley.edu}

\noindent Cynthia Vinzant, University of Michigan, Ann Arbor, MI 48109, USA,
 {\tt vinzant@umich.edu}

\end{document}